\documentclass[12pt]{amsart}
\usepackage{geometry, tikz-cd}                
\usepackage{svg}
\usepackage{graphicx}
\usepackage{amssymb}
\usepackage{epstopdf}
\usepackage{xcolor}
\usepackage[normalem]{ulem}
\usepackage{cancel}
\usepackage[percent]{overpic}
\usepackage{stackengine}
\usepackage{subfig}
\usepackage{wrapfig}
\usepackage{tabularx}
\usepackage{pinlabel}

\numberwithin{equation}{section}

\theoremstyle{plain}
\newtheorem{thm}{Theorem}[section]

\newtheorem{prop}[thm]{Proposition}
\newtheorem{lem}[thm]{Lemma}

\newtheorem{rem}[thm]{Remark}

\makeatletter
\@namedef{subjclassname@2020}{%
  \textup{2020} Mathematics Subject Classification}
\makeatother

\makeatletter
\def\moverlay{\mathpalette\mov@rlay}
\def\mov@rlay#1#2{\leavevmode\vtop{%
   \baselineskip\z@skip \lineskiplimit-\maxdimen
   \ialign{\hfil$\m@th#1##$\hfil\cr#2\crcr}}}
\newcommand{\charfusion}[3][\mathord]{
    #1{\ifx#1\mathop\vphantom{#2}\fi
        \mathpalette\mov@rlay{#2\cr#3}
      }
    \ifx#1\mathop\expandafter\displaylimits\fi}
\makeatother

\usepackage{geometry} 
\title{Distinguishing curve types and designer metrics}

\author{Ara Basmajian}
\address[Ara Basmajian]{The Graduate Center, CUNY \\ 365 Fifth Ave., N.Y., N.Y., 10016 and Hunter College, CUNY \\ 695 Park Ave., N.Y., N.Y., 10065, USA}
\email{abasmajian@gc.cuny.edu}
 \thanks{A.B. Partially supported by a PSC-CUNY Grant and Simons Collaboration Grant (359956, A.B.)}
\author{Sayantika Mondal}
\address[Sayantika Mondal]{The Graduate Center, CUNY \\ 365 Fifth Ave., N.Y., N.Y., 10016}
\email{}
\thanks{S.M. Partially supported by a Provost’s Pre-Dissertation Science Research Fellowship}

\begin{document}
\begin{abstract}
Let $\gamma$ be a filling curve on a topological surface $\Sigma$  of genus $g \geq 2$. The inf invariant of $\gamma$, denoted $m_{\gamma}$,  is the infimum of the length function on the space of marked hyperbolic structures on $\Sigma$. This infimum is realized at a unique hyperbolic structure,  $X_{\gamma}$, which we call the optimal metric associated to $\gamma$. In this paper, we investigate properties of the inf invariant and its associated optimal metric. Starting from a filling curve and a separating curve, we construct a two integer parameter family of curves for which we derive coarse length bounds and qualitative properties of their associated optimal metrics. In particular, we show that there are infinitely many pairs of filling curves, each pair having distinct $\text{inf}$ invariants but the same self-intersection number.

The inf invariants give rise to a natural spectrum, we call the 
inf spectrum, associated to  the moduli space of the surface. We provide coarse bounds for this spectrum.

\end{abstract}

\maketitle


\section{Introduction}
\label{sec: intoduction}

There are many topological invariants  one can associate with homotopy classes of closed curves. These include algebraic and geometric self-intersection numbers, intersection numbers  with curves in a class of curves (for example, simple ones),  the Goldman bracket, the number and type of complementary components, mapping class group stabilizers, and many others. How these invariants interact and determine the curve type (mapping class group orbit) is an active area of research.
In this paper, we begin an investigation of  the {\it \text{inf}  invariant}  (shortest length metric) associated to a filling curve,  its relationship with the  geometric self-intersection number,  and its relation to the optimal metric that is tailored to produce the minimum length.  While clearly the geometric self-intersection number is a type  invariant,   one of the questions we address is if the inf invariant can distinguish between curves that have the same self-intersection number.  Starting from a filling curve and a separating curve, we construct a two integer parameter family of curves for which we derive coarse length bounds and qualitative properties of their associated optimal metrics. In particular, we show that there are infinitely many pairs of filling curves, each pair having distinct $\text{inf}$ invariants but the same self-intersection number.  These invariants associated with each filling curve of the surface lead naturally to a spectrum, we call the {\it inf spectrum}, associated to the moduli space of the surface. We derive coarse bounds for this spectrum.

Let $\Sigma$ be a closed topological surface of genus $g \geq 2$, $T(\Sigma)$ its associated  Teichm\"uller space of marked hyperbolic structures of finite area,  and
 $\mathcal{M}(\Sigma)$  the moduli space of complete finite area hyperbolic structures on $\Sigma$. For   $\gamma$ a closed curve on 
$\Sigma$,  let   
$\ell_{\gamma}: T(\Sigma) \rightarrow \mathbb{R}_{\geq 0}$ be the length function of $\gamma$.   We define the $\text{inf}$ invariant of $\gamma$ to be,
$$
m_{\gamma}:=
\text{inf} \{\ell_{\gamma}(X): X \in T(\Sigma) \}.
$$
 This infimum is achieved at a unique marked hyperbolic structure denoted $X_{\gamma}$. We say that $X_{\gamma}$
is the {\it optimal metric}  associated to $\gamma$.

\vskip5pt
\noindent{\bf Theorem A.} (Theorems \ref{thm: sepoptimalmetric} and \ref{thm: separating curve construction})
{\it  Given a  closed surface of genus $g \geq 2$,  there exists an  infinite set of positive integers
 $\mathcal{K}$ so that for each 
 $k \in \mathcal{K}$  there are  a  pair of  curves 
 $(\alpha_{k}, \beta_{k})$ for which 
\begin{itemize}
\item $\alpha_{k}$ and $\beta_{k}$ are each  filling
\item  $\alpha_{k}$ and $\beta_{k}$  each have   self-intersection number $k$
\item  $m_{\alpha_k}\lesssim \log k < \sqrt{k}\lesssim  m_{\beta_k}$
\item the optimal metrics $\{X_{\alpha_{k}}: {k \in \mathcal{K}}\}$ go to the boundary stratum of 
$\mathcal{M}(\Sigma)$ corresponding to pinching   a  simple curve 
\item the optimal metrics $\{X_{\beta_{k}}: {k \in \mathcal{K}} \}$ stay in a compact part of $\mathcal{M}(\Sigma)$. \end{itemize}}

The construction we use to prove  Theorem A holds even if the surface has punctures as long as the Euler characteristic is less than -1, that is,  provided $\Sigma$ has an essential non-peripheral separating curve.  This will be investigated in a forthcoming paper by the second author.

Let $\gamma_{0}$ be a minimal filling curve (that is,
has self-intersection number $2g-1$), and let $\eta$ be a simple separating curve that intersects $\gamma_{0}$ at 2 points (see Lemmas 
\ref{lem: minimalfillingcurve} and \ref{lem:minimalandseparating} for more details). 
The curve $\eta$ separates $\Sigma$ into the two surfaces $\Sigma_{1}$ and $\Sigma_{2}$, and  
if $X \in \mathcal{T}(\Sigma)$ then we denote  the hyperbolic structure $X$ restricted to $\Sigma_{i}$ by 
$X_{i}$ for $i=1,2$.
In order to prove Theorem A,
we use $\gamma_{0}$ and $\eta$ to  construct a two  parameter curve family 
$\{\gamma_{m,n} :m,n \in \mathbb{N}\}$. 
We compute the self-intersection number of $\gamma_{m,n}$ and derive coarse length bounds. Namely

\vskip5pt
\noindent{\bf Theorem B.} (Proposition \ref{prop: separating} and Lemma \ref{lem:etagoestozero}) {\it 
Fix a point $p \in \Sigma$. let $\gamma_{m,n} \in \pi_1 (\Sigma,p)$
be the curve $\eta^{m}\ast  \gamma_{0}^{n}$.
\begin{enumerate}
\item 
  Fix $\epsilon >0$.
For all $X \in  \mathcal{T}_{\epsilon}(\Sigma)$, we have
\begin{equation*}
c_{1}n+c_{2}m + c_{3} 
\leq 
\ell_{\gamma_{m,n}}(X) 
\leq 
c_{4}n + c_{5} m +c_{6}
\end{equation*}
\item
For $n$ bounded above, $\ell_{\eta}(X_{\gamma_{m,n}}) \rightarrow 0$, as $m \rightarrow \infty$.
\item For $m$ bounded above, 
$\liminf_{n \rightarrow \infty} \ell_{\eta}(X_{\gamma_{m,n}})>0$
\end{enumerate}
where  $c_{i}=c_{i}(X)$ for $i=1,2,...,6$ is an explicit  positive constant that only depends on $X$.}

\vskip5pt

In  Theorem B
 
$$\mathcal{T}_{\epsilon} (\Sigma):=
\bigg\{X \in \mathcal{T}(\Sigma): \text{min} \big\{\textup{sys} (X_{1}),
\textup{sys}(X_{2})\big\}\geq \epsilon, 
\ell_{\eta} (X) \leq {1} \bigg\} $$
The notation $\textup{sys}(X_i)$ denotes the length of the shortest non-peripheral geodesic in $X_i$. 

Items (2) and (3) of Theorem B  are illustrative of  how the curve type influences its associated  optimal metric.
Namely, in item (2) of Theorem B the curve $\gamma_{m,n}$ is winding more and more around $\eta$ as $m$ gets large, forcing the length of $\eta$ with respect to the  optimal metric to be short. In item (3) of Theorem B the curve $\gamma_{m,n}$  intersects $\eta$ more and more as $n$ gets large, forcing the  length of $\eta$ with respect to the optimal metric 
to be bounded below.

We remark that the inf invariant, being a mapping class group invariant, is an invariant of the moduli space of the surface. On the other hand, there exist  distinct filling curve types for which their inf  invariants are the same.  For example, the well-known length equivalent curves \cite{leininger2003equivalent} can be realized as filling curves and thus are examples. In Arettines \cite{arettines2015geometry} a minimal intersecting filling curve is  explicitly found on a surface of genus $g$. Moreover,  the metric associated to this curve is explicitly realized  and the $\text{inf}$ invariant is given by $m_{\gamma}= (4g - 2)\text{arccosh}(2 \cos [\frac{\pi}{4g-2}] + 1)$. However in general,
given  a filling curve  $\gamma$  it is a difficult problem  to precisely determine  the 
$\text{inf}$   invariant  and the optimal metric  associated to $\gamma$. For this reason, we focus on coarse length bounds and qualitative properties of the optimal metric.

For $\gamma$ a filling curve, denote its orbit under the mapping class group  by $[\gamma]$, and 
its inf invariant by   $m_{[\gamma]}$. Associated to each topological surface,  we
 define the {\it \text{inf}  spectrum of } $\Sigma$  to be
the ordered multiset
$$
 \mathcal{I}Spec (\Sigma) := \{m_{[\gamma]}: \gamma \text{ filling on } \Sigma \}
 $$

 \vskip5pt 

\noindent{\bf Theorem C}. (Theorem \ref{thm: inf spec bounds})
  {\it  For any $L>0$, the set

$$\{ [\gamma]:\gamma \text{ filling on }\Sigma, m_{\gamma}\leq L \}  \text{ is finite.}$$
 More concisely,   there exists $L_{0}>0$ so that 

$$
e^{bL}
\leq
|\{ [\gamma]:\gamma \text{ filling on }\Sigma,m_{\gamma}\leq L \}|
\leq e^{ce^{L}} 
$$
for all $L \geq L_{0}$. Here $b, c >0$ are  explicit constants that only depend on the Euler characteristic $\chi(\Sigma)$.}

\vskip5pt
 
The finiteness in Theorem  C follows  from the fact that a closed geodesic  of length less than  $\frac{1}{2} \log{\frac{k}{2}}$ has  at most  $k$ self-intersections   
(\cite{basmajian2013universal}) and 
the  number of  mapping class orbits of curves with less than $k$ self-intersections is finite  
(\cite{cahn2018mapping}, \cite{sapir2016orbitsnonsimpleclosedcurves}).  To achieve the coarse bounds in Theorem C we use  results of Aougab and Souto  (\cite{souto2018counting}).

The fact that we consider only filling curves for the inf invariant is explained in Remark \ref{rem: not filling}. More quantitative versions of Theorems A and B  can be found in 
 Theorems \ref{thm: sepoptimalmetric}, \ref{thm: separating curve construction}, and Proposition \ref{prop: separating}. Some of the work in this paper goes beyond what is needed to prove the stated theorems in the introduction. For example, the coarse lengths on closed curves and arcs may be applied in various other contexts. We include these as they may be of independent interest. 

\subsection{Related results:}  
One can study  the relationship  between self-intersection  and length in various contexts including for non-filling curves.  One context is to fix the hyperbolic structure $X$ and ask for  the shortest closed geodesic with at least $k$ self-intersections. Next one can expand to considering all hyperbolic structures on a fixed topological surface, and lastly one can look for the shortest closed geodesic with $k$  self-intersections on any hyperbolic surface. For the relevant results see  
\cite{basmajianstable}, \cite{basmajian2013universal}, 
\cite{basmajian2017geometric}, 
\cite{cahn2018mapping},
\cite{ChasPhillipsoncepunctured}, \cite{Chasphillipstwice}, \cite{erlandsson2020short}, \cite{Gillmankeenwordsequences}, \cite{Malesteinputnam}, \cite{sapir2016boundnonsimp}, \cite{sertan} and \cite{torkaman2024intersectionpointsclosedgeodesics}.

In this paper, we fix a filling curve type and ask for the metric that yields the shortest length.  Papers involving metrics tailored to  curves and infimal length include \cite{aougabmetricsuited}, \cite{souto2018counting}, and 
\cite{Gasterinfima}.
Papers involving filling curves, their constructions, and their lengths include
\cite{aougab2022combinatoriallyrandomcurvessurfaces}, 
\cite{arettines2015geometry},
\cite{girondo2023minima}, 
 \cite{parlier2021topological},
 \cite{sapir2023lengthcomparisontheoremgeodesic},
 \cite{Voshortclosed},
\cite{wanglengthminoneholetorus}, 
and \cite{zhang2008hyperbolic}.

\subsection{Section plan and notation:}
\label{subsec: plan}

 The notation $f(k) \lesssim g(k)$ is used to mean that
there exist constants $c_1, c_2 >0$ such that
$f(k) \leq c_{1} g(k) + c_2$, for $k$ large.  Similarly,   $f(k) \gtrsim g(k)$ 
denotes $f(k) \geq  c_{1} g(k) -c_2$, for $k$ large. The constants in either one of these inequalities depend only on the topology of the surface and a fixed chosen filling curve. We use $\Sigma$ to denote a topological surface, and
a surface with a hyperbolic metric is usually denoted by a capital latin letter such as $X$ or $Y$. To avoid cluttered notation, we occasionally use a small letter to denote both a curve and its length. It  should be clear from the context.   In Table \ref{table:1}  we list commonly used definitions with the first section they appear in and their notation. 

Section 2 is devoted to elementary facts  on the topology and geometry of filling curves on hyperbolic surfaces. 
In Section 3 the inf invariant is introduced as well as results on orthogeodesics and necessary geometric lemmas that will be needed later to achieve bounds on various curve types.  In Section 4 we construct a two parameter curve family, compute self-intersection numbers,  and supply  coarse length bounds for the curves in the family. In Section 5 properties of the optimal metrics associated to the two parameter family of curves are derived  and in Section 6 we supply coarse bounds on   the infimal length of  curves in the family with the same self-intersection number. Finally, in Section 7 we give bounds for the growth of the inf spectrum and conclude with final remarks.

\begin{center} \begin{table}
	\begin{tabular}{|c|c|c|}
		\hline \hline 
		{\bf Definition} & {\bf Section} &{\bf Notation} \\
		\hline
		\hline 
		\hline
        2-parameter curve family, 
        $\eta^{m}\ast \gamma_{0}^{n}$& \ref{sec: intoduction} &$\gamma_{m,n}$\\
		\hline
        Systole length  of $X$ & \ref{sec: intoduction}& $\text{sys}(X)$\\
		\hline
		Asymptotic upper bound& \ref{subsec: plan} &$f(k) \lesssim g(k)$ \\
		\hline
		Asymptotic lower  bound& \ref{subsec: plan} &$f(k) \gtrsim g(k)$ \\
        \hline
		Euler Characteristic of $\Sigma$ & \ref{sec: Basics} &$\chi := \chi(\Sigma)$ \\
        \hline
        
		Standard  collar function& \ref{sec: Basics} &$r(\frac{x}{2})$ \\
		\hline
		Standard collar neighborhood& \ref{sec: Basics}  & $\mathcal{N}(\cdot)$ \\
		\hline
		Teichm\"uller Space&\ref{sec: Basics}  & $\mathcal{T}(\Sigma)$  \\
		\hline 
		
		Length function of $\gamma$ on Teichm\"{u}ller space & \ref{sec: Basics}  & $\ell_{\gamma}$  \\
		\hline
		Length function of $\gamma$ on moduli space & \ref{sec: Basics}  & $L_{\gamma}$  \\
		\hline
		Marked hyperbolic structure & \ref{sec: Basics} &$(X,f)$   \\
		\hline
		Mapping class group&\ref{sec: Basics}& $MCG(\Sigma)$\\
		\hline
		Intersection number& \ref{sec: Basics} & $i(\cdot ,\cdot)$\\
		\hline
        $\epsilon$-Teichm\"uller space& \ref{sec: Min  invariants} & $T_{\epsilon}(\Sigma)$  \\
	        \hline
		Inf  invariant& \ref{sec: Min  invariants} & $m_{\gamma}$ \\
		\hline
		Moduli space& \ref{sec: Min  invariants}& $\mathcal{M}:=\mathcal{M}(\Sigma)$  \\
		\hline
		Optimal metric for $\gamma$& \ref{sec: curve family} & $X_{\gamma}$ \\
        \hline
        \hline
	\end{tabular} \caption{Notations} \label{table:1}
\end{table}\end{center}

\section{Basics} \label{sec: Basics}

Let $\Sigma=\Sigma_{g,n}$ be an orientable topological surface 
of genus $g$ and $n$ points removed. We denote the Euler characteristic of $\Sigma$ by 
$\chi :=\chi (\Sigma)$. We assume that 
$\Sigma$ supports a hyperbolic structure with non-abelian  
fundamental group. We define the mapping class group as
 $\text{MCG}(\Sigma):=\text{Diffeo}^{+} (\Sigma)/\sim$, where
 $f \sim g$ if $f \circ g^{-1}$  is isotopic to the identity.

 \vskip10pt
 \noindent{\bf Topology:}
\vskip10pt
 
A closed curve  in $\Sigma$ is {\it peripheral}  if it bounds a punctured disc. The closed curve is {\it essential} if it is not homotopic to a point or is not peripheral. The {\it self-intersection number}  of a closed curve is the minimal number of intersections with itself  when in general position. If the underlying surface has a hyperbolic metric then the self-intersection number 
$i(\alpha, \alpha)$ is realized by the geodesic in the homotopy class of the curve. For a compact surface $\Sigma$ with boundary,  an {\it arc} is a 
 continuous mapping $\delta : [0,1] \rightarrow \Sigma$
where $\delta (0)$ and $\delta (1)$ are in $\partial \Sigma$.
An arc is {\it essential} if it is not homotopic 
(rel $\partial \Sigma$) to an arc on $\partial \Sigma$. 
For the purposes of this article, a {\it multi-curve}  ({\it multi-arc}) is a finite union  of curves (arcs), and the multi-curve (arc) is essential if its components are essential. We extend the definition of hyperbolic length to multi-curves and multi-arcs in the natural way.  A closed curve on 
$\Sigma$  is 
 {\it filling} if its complementary components are discs or punctured discs.  Equivalently every essential simple closed non-peripheral closed curve intersects it.  Similarly, an arc from the boundary of a compact surface to itself is said to be {\it filling} if every non-peripheral simple closed curve intersects it. Two curves 
 $\gamma_1$ and $\gamma_2$ are of the same {\it topological type}  if there exists an element  $g \in \text{MCG}$ so that $g(\gamma_1)$ is homotopic
 to $\gamma_2$. Unless we state otherwise our curves and arcs are  oriented.

 \begin{lem}
 Assume  $2-2g-n < 0$. Then the number of  mapping class group orbits of  (unoriented) separating simple curves on 
   a surface of genus $g$ with $n$ points removed is 
   \begin{itemize}
 \item $\lfloor \frac{g}{2} \rfloor$, if $g \geq 2$ and $n=0$. 
 \item $\lfloor \frac{n}{2}\rfloor -1$, if $g=0$ and $n\geq 3$.
 \item  $(n-1) +\lfloor \frac{g}{2}\rfloor (n+1)$,  if $g\neq 0$ and $n \neq 0$
 \end{itemize}
 
 \begin{proof} 
 
  The separating curves bound two subsurfaces each having one boundary component and genus at least one. The problem reduces to counting the number of ways of counting unordered admissible  partitions of $(g,n)$.   An admissible  partition of the pair $(g,n)$ is defined  to be  
 $(g_{1}, n_{1})$  and $(g_{2}, n_{2})$, where 
 $g=g_{1}+g_{2}$, $n=n_{1}+n_{2}$, and $(g_{i},n_{i})$ is not equal to 
 $(0,0)$ or $(0,1)$. 
 \end{proof}
 \end{lem}
 \vskip10pt
 
 \noindent{\bf Filling Curves:}
 
 \begin{lem}[Minimal filling curve] \label{lem: minimalfillingcurve}
  If $\gamma$ is a filling curve on $\Sigma_{g}$ then 
  $i(\gamma,\gamma) \geq 2g-1 $. Moreover, there exists a curve that realizes this lower bound. Such a curve is called a minimal filling curve. 
 \end{lem} 
 
 The lower bound and the existence of the curve   was shown in Arettines thesis\cite{arettines2015geometry}.

 \begin{rem} A result of Hass-Scott \cite{hass1985intersections} says that an immersed curve has minimal self-intersection number if it does not bound a monogon nor a bigon. 
 \end{rem}

 \begin{lem}\label{lem:minimalandseparating}
Given a surface $\Sigma_{g}$, there exists a minimal filling curve and a separating curve such that they intersect each other twice.
 \end{lem}
 \begin{proof}
     We provide an explicit construction for $\Sigma_{g}$ that allows us to see a minimal filling curve as a boundary of a disc/polygon and on this polygon identify a separating curve with intersection number 2 with the boundary.

     We begin with the case of a $g=2$. Consider the 12-gon on the left in Figure \ref{fig: polygon genus 2} with the identifications. As shown on the right in Figure \ref{fig: polygon genus 2} this is a genus two surface. 
     Now, consider the boundary of this polygon, this gives us a curve on the surface with intersection number 3. Since the complement of this curve is a disc by construction it is a filling curve. This is a minimal filling curve as it has intersection number 3 (\cite{arettines2015geometry}).

     Next consider the red edges shown in Figure \ref{fig:minimal sep intersection 2} on this polygon. Identifying these red edges yields a simple closed separating curve on the surface whose intersection number is 2 with the boundary curve (minimal filling curve).

     For a genus g surface this construction can be extended as shown in Figure \ref{fig:genus 3 polygon} and Figure \ref{fig:genus g polygon}, where we introduce four new edges and 2 new vertices (intersection points for the boundary curve) to our existing polygon to create a surface with one extra genus. Note that the central octagon gives us a once punctured torus, and as we move away from the octagon each pair of consecutive hexagons forms a twice punctured torus. At each step, we introduce two new self-intersections for the boundary curve, thus the self intersection number of the boundary curve is 2g-1 by induction. So this after identification of the boundary is a minimal filling curve on the surface.
     Note that the red edges after identification still form a separating curve, regardless of the genus of the surface and has intersection number 2 with the identified boundary curve (minimal filling).

 \end{proof}
 
 \begin{figure}[!tbp]
  \centering
  \subfloat[Genus 2 polygon]{\includegraphics[scale=0.23]{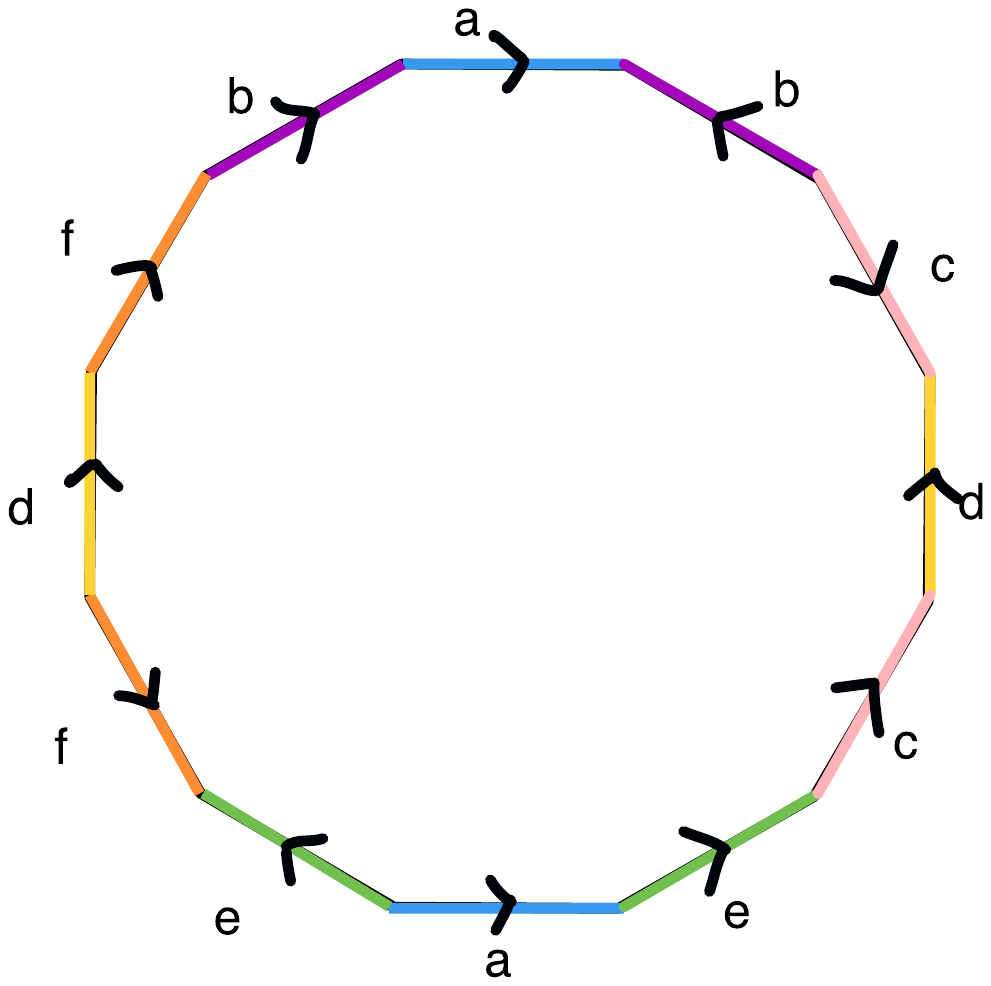}\label{fig:genus 2 polygon}}
  \hfill
  \subfloat[Genus 2 folded]{\includegraphics[scale=0.23]{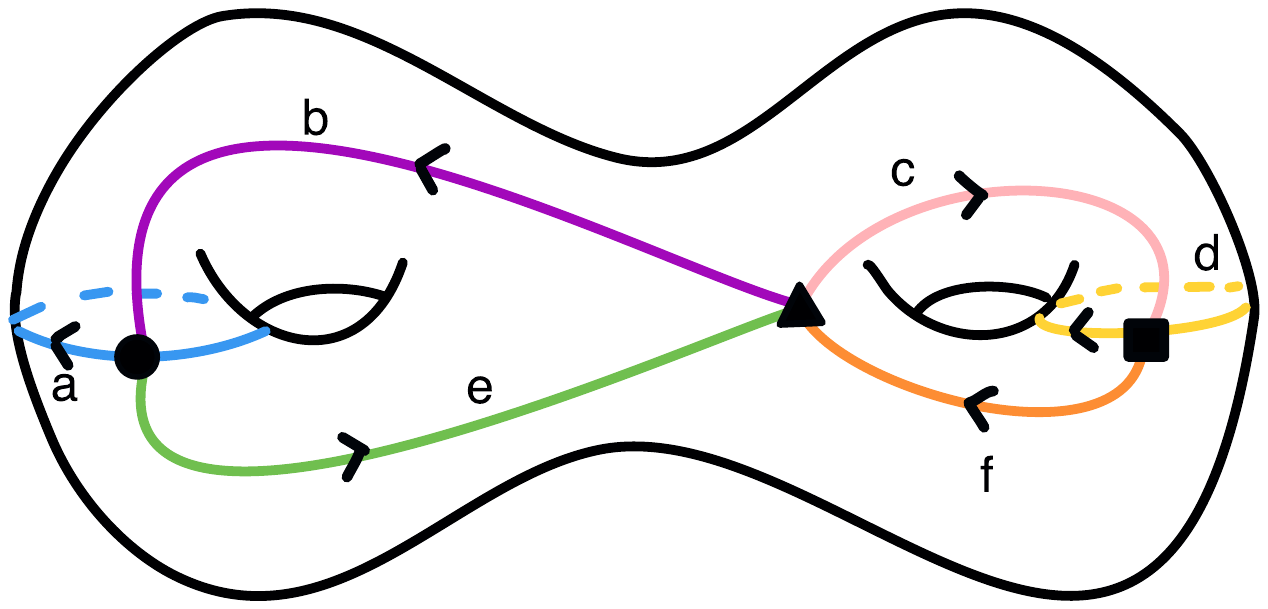}\label{fig:genus 2 folded}}

\caption{Polygonal representation of a genus 2 surface.}
\label{fig: polygon genus 2}
\end{figure}

\begin{figure}

\includegraphics[scale=0.3]{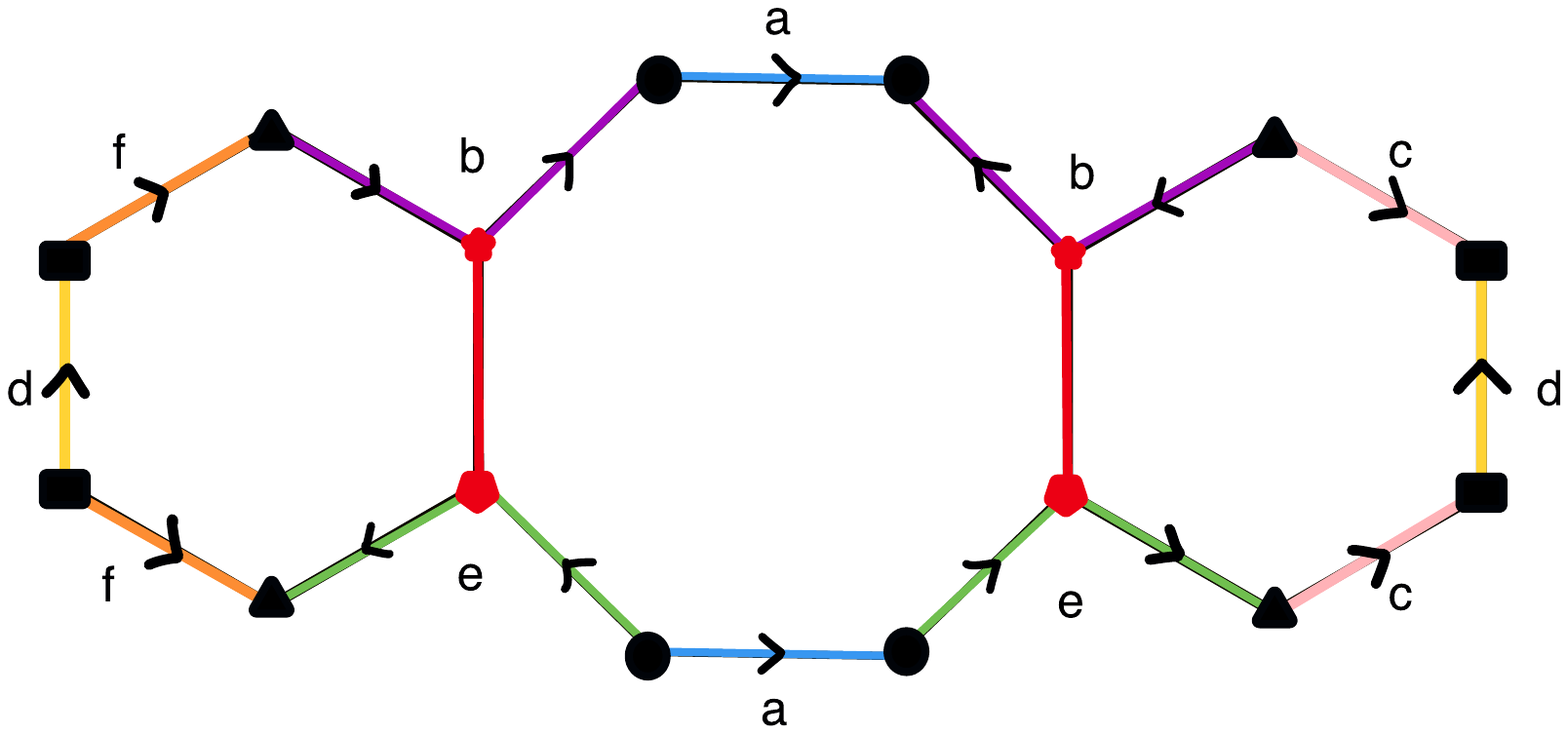}
\vspace{-30pt}
\caption{Separating curve on a genus 2 surface with intersection number 2 with minimal filling curve.}
\label{fig:minimal sep intersection 2}
\end{figure}
\begin{figure}
\includegraphics[scale=0.4]{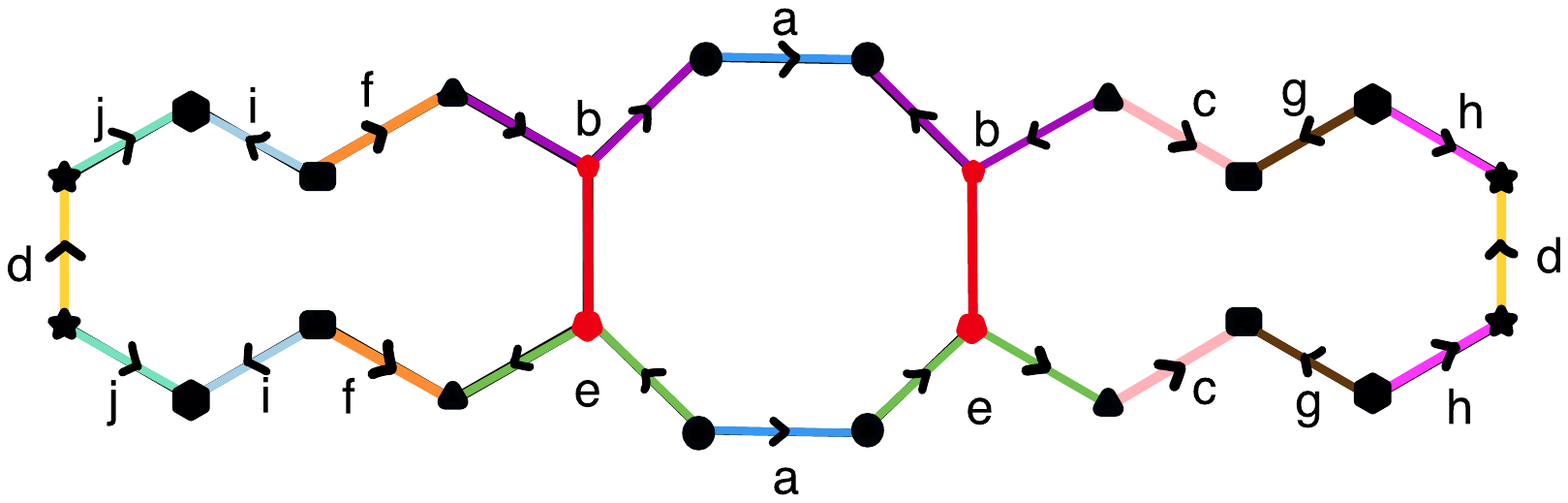}
\vspace{-70pt}
\caption{Polygonal representation for a genus 3 surface.}
\label{fig:genus 3 polygon}
\end{figure}

\begin{figure}
\includegraphics[scale=0.45]{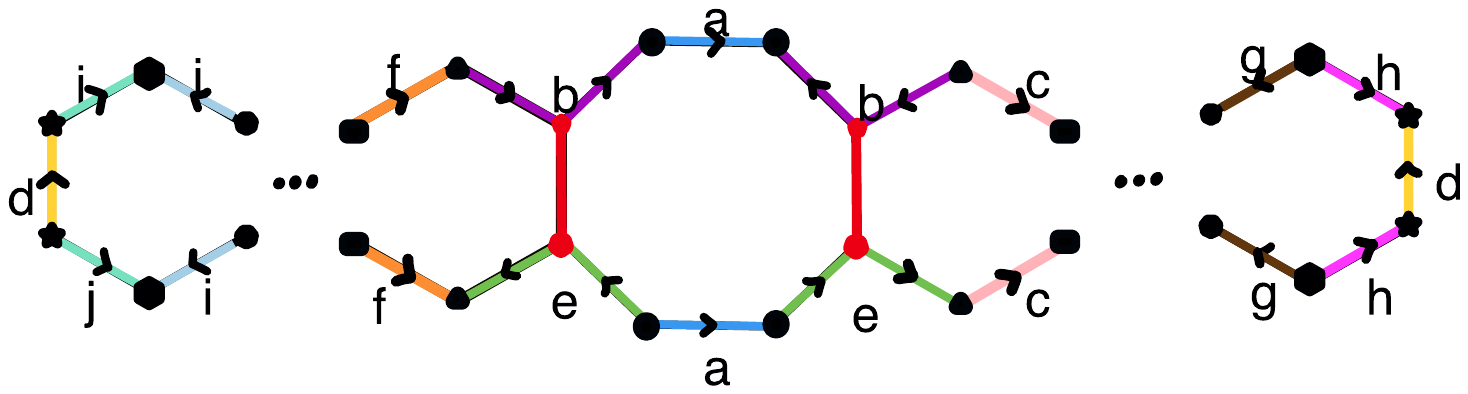}
\vspace{-80pt}
\caption{Polygonal representation for a genus g surface.}
\label{fig:genus g polygon}
\end{figure}

\vskip10pt
\noindent{\bf Geometry:}
\vskip10pt

A {\it marked hyperbolic structure} on $\Sigma$ is
an orientation preserving diffeomorphism   
$f:\Sigma \rightarrow X$, denoted $(X,f)$, where $X$ is a hyperbolic structure of finite area (no funnels). Two marked hyperbolic structures $(X,f)$ 
and $(Y,g)$ are equivalent if  
$g \circ f^{-1}: X \rightarrow Y$ is homotopic to an isometry. The Teichm\"uller space of $\Sigma$, denoted $\mathcal{T}(\Sigma)$,  is the space  of all marked finite area complete hyperbolic structures endowed with the Teichm\"uller metric.  When there  is no chance of confusion or we are not interested in the marking, we denote the marked hyperbolic surface $(X,f)$ as just $X$.

 Fix $\gamma$ an isotopy class of essential curve, and define
 the length function of $\gamma$ to be $\ell_{\gamma} : \mathcal{T}(\Sigma) \rightarrow \mathbb{R}_{+}$, where
 $\ell_{\gamma} (X,f)$ is the length of the $X$-geodesic homotopic to $f(\gamma)$. The mapping class group acts effectively on  $\mathcal{T}(\Sigma)$ by,
  $g \cdot (X,f) \mapsto   (X,f \circ g^{-1}).$
 Moreover, $\ell_{\gamma} (g^{-1}\cdot \mathcal{T}(X,f))=\ell_{g(\gamma)}(X,f)$.  The quotient of the Teichm\"uller space by the mapping class group is the {\it moduli space} of  finite area  hyperbolic metrics on 
 $\Sigma$, denoted $\mathcal{M}(\Sigma)$. The mapping class group acts as a  group of isometries on Teichm\"uller space both with respect to the Teichm\"uller metric and the Weil-Petersson metric (see  \cite{masur2009geometry},
 \cite{wolpert1987geodesic} for definitions and properties of these metrics). 

 We define the {\it systole} of a hyperbolic structure $X$ (on a possibly disconnected surface) to be the shortest non-peripheral closed geodesic (there may be more than one) on $X$. It's length is denoted
 $\text{sys}(X)$.

 \begin{lem}
Suppose $X$ is a finite area hyperbolic surface with possibly geodesic boundary.
Then there is a constant $C$ that only depends on the topology of $X$ so that 
$\text{sys}(X) \leq C.$    
 \end{lem}

 \begin{proof}
Let $\gamma$ be a geodesic on $X$  that realizes the systole length say $\ell$.

If $X$ has no boundary, then 
the convex core $C(X)=X$, and an area argument guarantees that $\ell$ is bounded by a constant that only depends on the area of $X$ (topology).  

Now if $X$  has geodesic boundary
then form the  surface  $\bar{X}$ by doubling along the geodesic boundary of $X$. Note that there is a natural reflection isometry $R: \bar{X} \rightarrow \bar{X}$.
Let $p$ be a point on $\gamma$, and consider the disc $D$  of radius 
$\frac{\ell}{2}$ centered at $p$.
We claim that $D$ is embedded in $\bar{X}$. This is clear if
$D \subset C(X)$. Otherwise, let 
$D_{1}=D \cap X$, and $D_{2}=R(X)\cap D$, and note that $D=D_{1} \cup D_{2}$.  By possibly switching the roles of $D_{1}$ and $D_{2}$ we may assume 
$\text{area}(D_{1}) \leq 
\text{area}(D_{2})$. Hence  $D_{2} \subset R(D_{1})$ and thus  $D$ is embedded in $\bar{X}$ and the conclusion follows. 
 \end{proof}

 \begin{rem} \label{rem: not filling}
 If $\gamma \subset \Sigma$ is not filling, then there exists an
 essential  simple closed curve on $\Sigma$ disjoint from $\gamma$. Cutting along it and noting that Thurston's strip map Theorem 
 ( see \cite{ParlierlengthRiemann} or \cite{thurston1998minimal}) guarantees that 
 $\ell_{\gamma}(X)$ can be decreased by a deformation that decreases the length of the boundary. This implies that the length function
 $\ell_{\gamma}$ cannot achieve its minimum in $\mathcal{T}(\Sigma)$
 but rather on the boundary $\partial \mathcal{T}(\Sigma)$. For this reason we assume  that $X$ has finite area (no funnels), and our curves are filling.
 \end{rem}
 
 We define the following function which will arise in our later coarse bounds for lengths. 
 $$
 r(x)=\text{arcsinh}\left(\frac{1}{\sinh x} \right).
 $$
 
 It is a consequence of the collar lemma (\cite{buser2010geometry})
 that a simple closed geodesic of length $\ell$ on a hyperbolic surface has a {\it natural (standard) collar neighborhood} of width 
 $r\left(\frac{\ell}{2}\right)$, and two such disjoint geodesics have disjoint collars. That is the distance from a point in the collar to the geodesic  is less than 
  $r\left(\frac{\ell}{2}\right)$.


\section{Inf invariants} \label{sec: Min  invariants}

Let $\Sigma$ be a surface of genus $g\geq 2$, and fix $\alpha$ a filling curve on $\Sigma$. We define two invariants of a curve or arc.  Both invariants can be defined in the more general setting of geodesic currents which we choose not to do as our interest in this paper is with constructing curves having certain geometric and topological properties. See  \cite{bonahon1988geometry}, \cite{hensel2023projection} for more on currents. In particular, in \cite{hensel2023projection}  the infimum length invariant for currents and their associated metrics arises.

Define the function:
$L_{\gamma} : \mathcal{M}(\Sigma) \rightarrow  \mathbb{R}_{\geq 0}$ to be, 
$X \mapsto \text{inf }\{ \ell_{\gamma} (g^{-1} \cdot (X,\phi)): 
g \in \text{MCG}\}$. Clearly, the choice of the marked structure  
$(X, \phi)$ does not effect the infimum. 
Define the {\it inf  invariant} of $\gamma$ to be: 
$m_{\gamma}:=
\text{inf }\{ \ell_{\gamma} (X, \phi): (X,\phi) \in \mathcal{T}(\Sigma)\}$.
Note that 
$m_{\gamma}=\text{inf }\{L_{\gamma}(X): X \in \mathcal{M}(\Sigma)\}$.

\begin{lem}[] Let  $\gamma$ be a filling curve on $\Sigma$. Then
 \begin{enumerate}
 \item  For $g \in \text{MCG}$, $L_{g(\gamma)}(X)=L_{\gamma}(X).$
 \item The infimum in the definition of  $L_{\gamma}(X)$ is achieved, and 
 $L_{\gamma}$ is a continuous function on 
  $\mathcal{M}$ with respect to its natural topology. 
\item $L_{\gamma}(X)  \rightarrow \infty$, as $X\rightarrow \partial\mathcal{M}$.
 \item  For $g \in \text{MCG}$, $m_{g(\gamma)}=m_{\gamma}.$
\item The infimum in the definition of  $m_{\gamma}$ is achieved at a unique marked hyperbolic structure. Hence, $L_{\gamma}$ achieves its infimum at a unique $X$ in $\mathcal{M}$.
\end{enumerate}
\end{lem}
\begin{proof} 
Item (1) follows from the definition. To prove item (2), first note that by elementary properties of the length function $\ell_{\gamma}$, 
$$L_{\gamma}(X)=\text{inf }\{ \ell_{\gamma} (g^{-1} \cdot (X,\phi)):g \in \text{MCG}\}
=\text{inf }\{ \ell_{g(\gamma)} ((X,\phi)): g \in \text{MCG}\}.
$$
Hence, $L_{\gamma}(X)$ is the same as  the infimum over 
the mapping class group orbit of $\gamma$ on $X$. Since the number of curves of $X$-length less than a real number is finite, (see 
\cite{buser2010geometry}) we can conclude that the infimum is achieved.  The continuity of 
$L_{\gamma}$ follows from the fact that the length function $\ell_{\gamma}$ is a continuous function on Teichm\"uller space
and  that  if two metrics are close in 
moduli space then there is a mapping class group orbit point of one near the other. Item (3) is a consequence of the 
Mumford compactness theorem.

To prove  item (4),  if  $g \in \text{MCG}(\Sigma)$ then 

\begin{align*}
\label{}
  m_{g(\gamma)}&=\text{inf }\{ \ell_{g(\gamma)} (X, f): (X,f)  \in \mathcal{T}(\Sigma)\} \\
    &=\text{inf }\{ \ell_{\gamma} (X, f \circ g): (X,f)  \in \mathcal{T}(\Sigma)\}\\
      &=\text{inf }\{ \ell_{\gamma}  ( (X, f \circ g)): g^{-1}(X,f)  \in \mathcal{T}(\Sigma)\}\\
&=\text{inf }\{ \ell_{\gamma}  ( (X, f \circ g)): (X,f \circ g)  \in \mathcal{T}(\Sigma)\}\\
&=m_{\gamma}
\end{align*}

To prove item (5),  first  note   that $m_{\gamma}$ and $L_{\gamma}$ are related by the formula,
$m_{\gamma}=\text{inf}\{L_{\gamma}(X): X \in \mathcal{M}\}$. Now 
suppose $L_{\gamma}(X_{n}) \rightarrow m_{\gamma}.$ Then the sequence of metrics $\{X_n\}$ in $\mathcal{M}$ must stay in a compact subspace by the Mumford compactness theorem. 
Hence, there exists a subsequence that converges to a metric 
$X \in \mathcal{M}$. We continue to use the same notation for the subsequence. Now the continuity of $L_{\gamma}$ guarantees that 
$L_{\gamma}(X_{n}) \rightarrow L_{\gamma}(X)$, and therefore 
$m_{\gamma}=L_{\gamma}(X)$. For uniqueness,  note that between any two marked hyperbolic structures there is a unique Weil-Peterson geodesic and the length function $\ell_{\gamma}$ is convex along it 
(see \cite{wolpert1987geodesic}).  Now if there were two marked structures for which $\ell_{\gamma}$ achieves the minimum it would have to be constant along the W-P geodesic. 
But  since 
$\gamma$ is filling $\ell_{\gamma}$ can not be constant along this geodesic. Thus the optimal metric for $\gamma$ must be unique. 
\end{proof}

The existence and uniqueness of the optimal metric  also follows from \cite{bonahon1988geometry}, \cite{sapir2016bounds} in the generalized setting of currents.

\vskip10pt

\noindent{\bf Filling arcs:}
\vskip10pt

Let $\Sigma_{1}$ be a finite type surface with one boundary component which we call $\eta$. For  $\epsilon >0$,  we consider the subspace of Teichm\"uller space with bounded geometry away from the boundary of 
$\Sigma_{1}$, and with boundary  length upper bounded. 
 More precisely this subspace $\mathcal{T}_{\epsilon}(\Sigma_{1})$ is 

  \begin{equation}
{\mathcal{T}_{\epsilon} (\Sigma_{1})}=
\bigg\{Y \in \mathcal{T}(\Sigma_{1}): 
\text{sys}(Y) \geq \epsilon, \ell_{\eta}(Y) \leq {1}\bigg\}
\end{equation}

If $\Sigma$ is a topological surface with no boundary, and  $\eta$ is a separating curve whose complementary subsurfaces are denoted $\Sigma_1$ and $\Sigma_2$. Then, we say that 
$(X,f) \in {T_{\epsilon}(\Sigma)}$ if $(X,f) \in 
{T_{\epsilon}(\Sigma_{1})}$ and    ${T_{\epsilon}(\Sigma_{2})}$.


\begin{figure}
\includegraphics[scale=0.35]{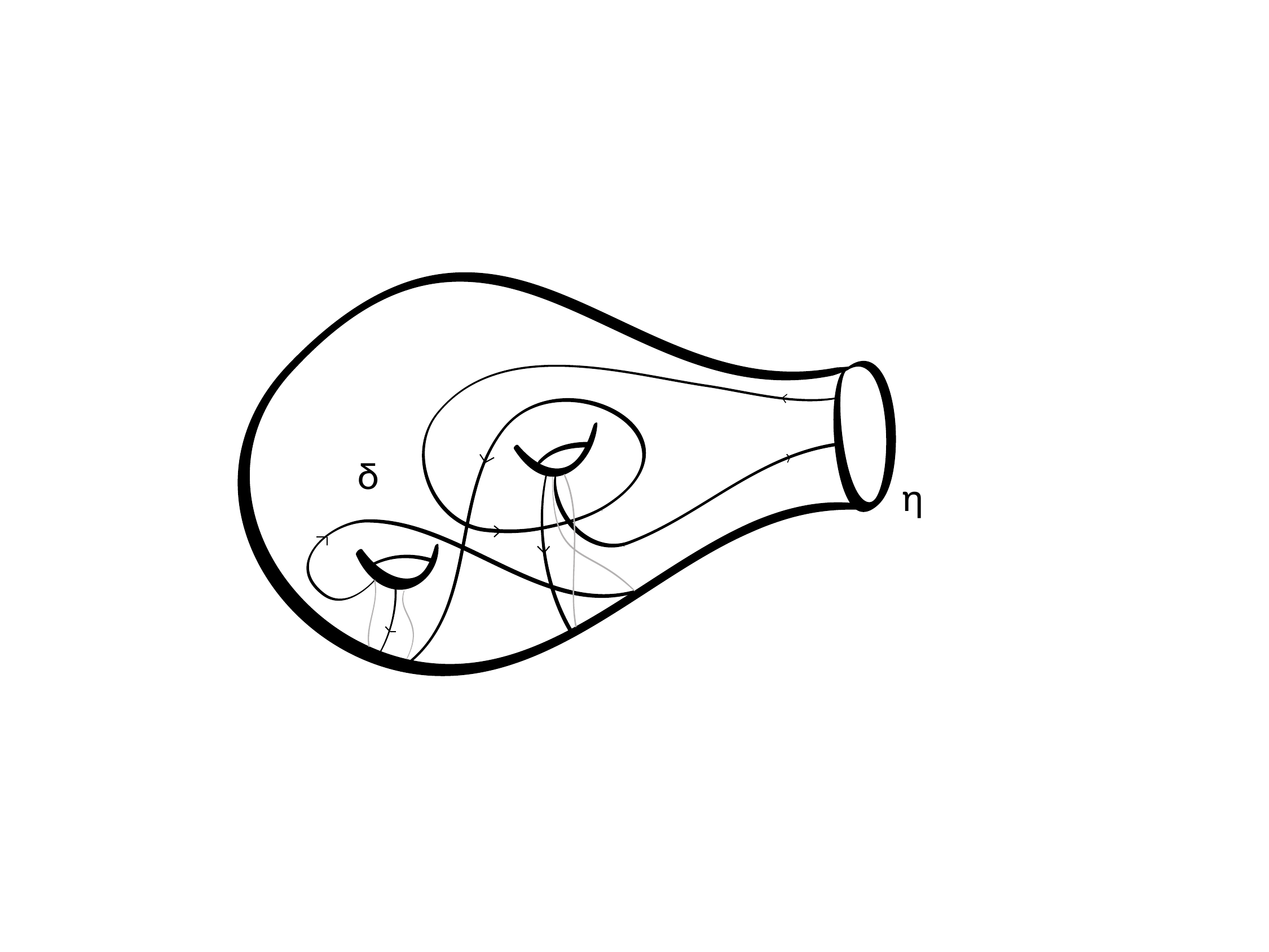}
\vspace{-80pt}
\caption{An orthogeodesic  $\delta$ from $\partial Y$ to itself}
\label{fig:delta}
\end{figure}

Suppose $Y$ is a compact hyperbolic surface with one geodesic boundary component and let $\delta$ be an orthogeodesic from $\partial Y$ to itself. Hence $\delta$  starts orthogonal to $\eta$, passes through the collar about $\eta$,  then traverses $Y$, and then returns to $\eta$ with its final segment passing through the collar about $\eta$. Denote by 
$\delta^{\prime}$  the orthogeodesic  that  traces the same path as $\delta$ except the initial and terminal segments in the collar about $\eta$ are deleted (see Figures \ref{fig:delta}  and \ref{fig:deltaprime}).


\begin{figure}
\includegraphics[scale=.35]{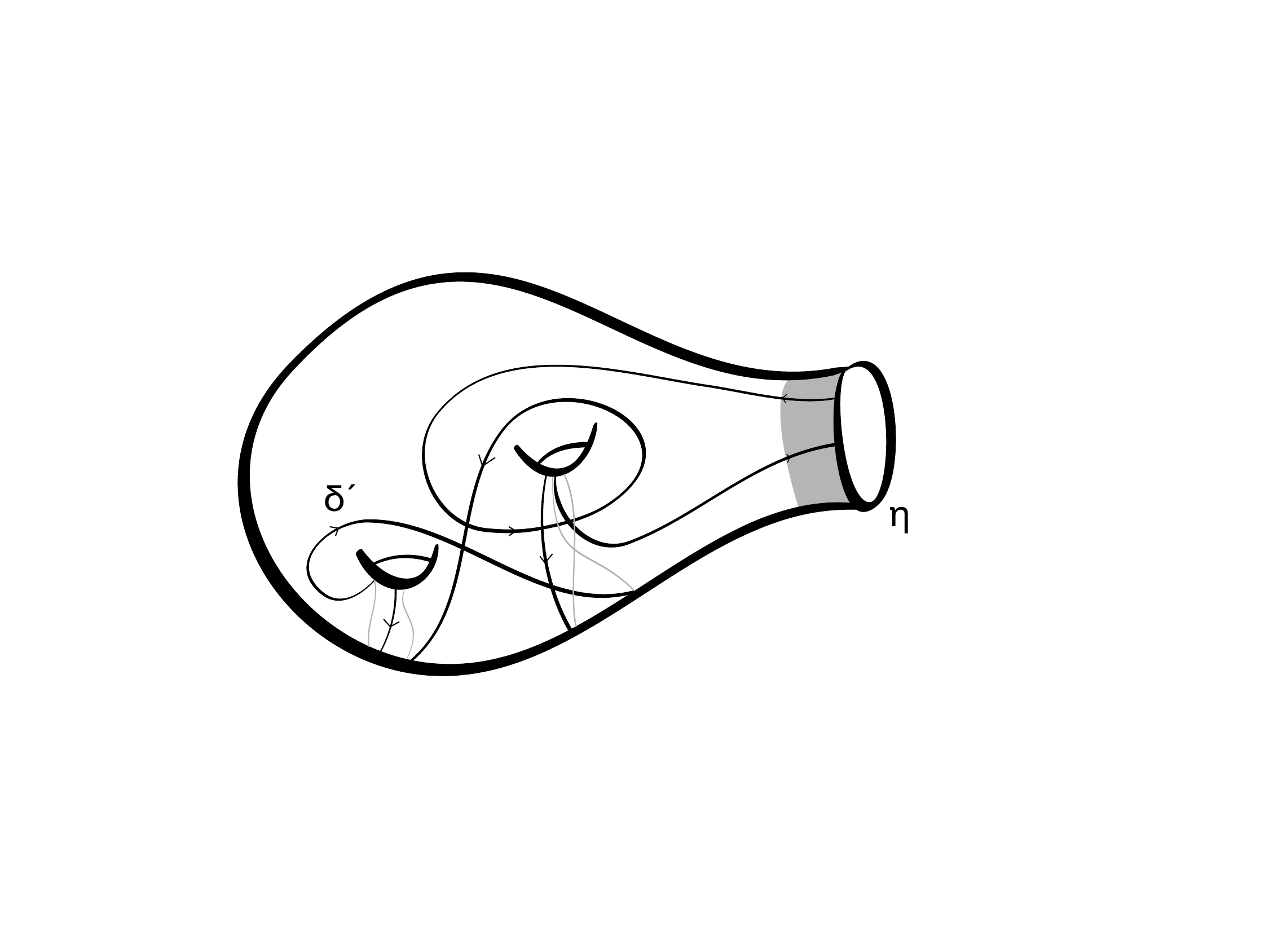}
\vspace{-80pt}
\caption{The orthogeodesic $\delta^{\prime}$ from $\partial \text{(Collar)}$ to itself}
\label{fig:deltaprime}
\end{figure}

\begin{lem}\label{lem: arc bounded} 
Fix  $\epsilon >0$ and let   $\delta$ be  a multi-arc  from 
$\partial\Sigma_{1}$ to $\partial\Sigma_{1}$. 
 There exist constants, $d,D>0$ depending only on $\epsilon$  and the (relative free) homotopy class of $\delta$  so that,
\begin{equation}
d \leq \ell_{\delta^{\prime}} (Y) \leq D, \text{  for all }Y \in T_{\epsilon}(\Sigma_{1}).
\end{equation}
\end{lem}

\begin{rem}
The upper bound $\ell_{\partial \Sigma_{1}}(Y) \leq 1$ in 
$T_{\epsilon}(\Sigma_{1})$
 is necessary to guarantee an upper bound on $\ell_{\delta^{\prime}}$. This is because  $\delta^{\prime}$ may wrap around $\partial \Sigma_{1}$ at some intermediate portion of its path. 
\end{rem}

\begin{proof} It's enough to prove the lemma for a connected  arc since 
everything in consideration extends linearly to multi-arcs. 
The function, $\ell_{\delta}$ is continuous on $\mathcal{T}(\Sigma_{1})$, in particular on 
$\mathcal{T}_{\epsilon}(\Sigma_{1})$ \cite{basmajian2025orthosystolesorthokissingnumbers}.  The arc $\delta$ initially starts orthogonally to $\eta$ and ends  orthogonal to $\eta$. Along the way it may enter the collar about $\eta$ a number of times. Then since 
$\ell_{\delta^{\prime}}=\ell_{\delta}-2r({\frac{\ell_{\eta}}{2}}),$ we may conclude that  $\ell_{\delta^{\prime}}$ is a continuous function
on  $T_{\epsilon}(\Sigma_{1})$.
 The only way a sequence in  
$T_{\epsilon}(\Sigma_{1})$ can go off to infinity is if 
$\ell_{\eta} \rightarrow 0$, but in that case the length of the geodesic arc $\delta^{\prime}$ approaches the length of the geodesic arc $\delta^{\prime}$ on the cusped surface. The limiting hyperbolic structure is one where 
$\eta$ becomes a cusp.
 Thus $\ell_{\delta^{\prime}}$ is bounded from above and below.
\end{proof}

Let $\eta$ be a separating curve on $\Sigma$, and let 
$\Sigma_1$ be one of the components of $\Sigma-\{\eta\}$.  
Suppose $\gamma$ is a filling curve on $\Sigma$ and $\delta$ is the multi-arc defined by $\gamma$ in $\Sigma_{1}$. In the presence of a hyperbolic metric $X_1$, while the $X_1$-geodesic of $\gamma$ may not 
intersect the $X_1$-geodesic of $\eta$ orthogonally, we nevertheless have

\begin{lem} For $\epsilon>0$, there exists $K=K(\epsilon)$ so that for any hyperbolic structure $X_1$ on $\Sigma_1$, 
with  $\ell_{\eta}(X_1) \leq \epsilon$
\begin{displaymath}
 \ell_{\delta^{\prime}}(X_1) \leq a^{\prime}   \leq  K+ \ell_{\delta^{\prime}}(X_1)
\end{displaymath}
where $a$ is the geodesic multi-segment defined by the 
$X_1$-geodesic of $\gamma$ restricted to $\sigma_1$ , $a^{\prime}$ is the geodesic  $a$  minus its initial and terminal segments in  the collar of $\eta$, and 
$\delta^{\prime}$ is the multi-orthogeodesic as discussed previously. 
\end{lem}

\begin{proof} The first inequality follows from the fact that $\delta^{\prime}$
is orthogonal to $\partial \mathcal{N}(\eta)$ and is freely homotopic to 
$a^{\prime}$. For the second inequality, note that $a^{\prime}$  is at most the piecewise geodesic segment given by going along the boundary of the collar, traversing the orthogeodesic $\delta^{\prime}$ and then traveling along the boundary of the collar of 
$\eta$ again. Since the boundary of the collar has length bounded from above by a function of $\ell_{\eta}(X_1)$ which is less than $\epsilon$, we are done. 

\end{proof}

 We fix  $X$  a hyperbolic surface and $\eta$ a separating simple closed geodesic on it.  Note that $\eta$ has a maximal collar, and that it is not necessarily symmetric. That is,
the width on one side may be different than the width on the other.  Denote by
 $\mathcal{C}$ this maximal collar. Let $a$ be a geodesic arc that joins the two components of $\partial \mathcal{C}$.  Now consider the geodesic  $h$ in  $\mathcal{C}$ which is perpendicular to $\eta$ and starts at the start point of $a$ and ends at the other component of  $\partial \mathcal{C}$.  We say that $a$ {\it winds  $m$-times  around $\eta$ in $\mathcal{C}$}  if  $a$ intersects $h$ 
$m$-times in the interior of $\mathcal{C}$. See the left most image in Figure \ref{fig:unrolling}.


\begin{figure}

\includegraphics[scale=.38]{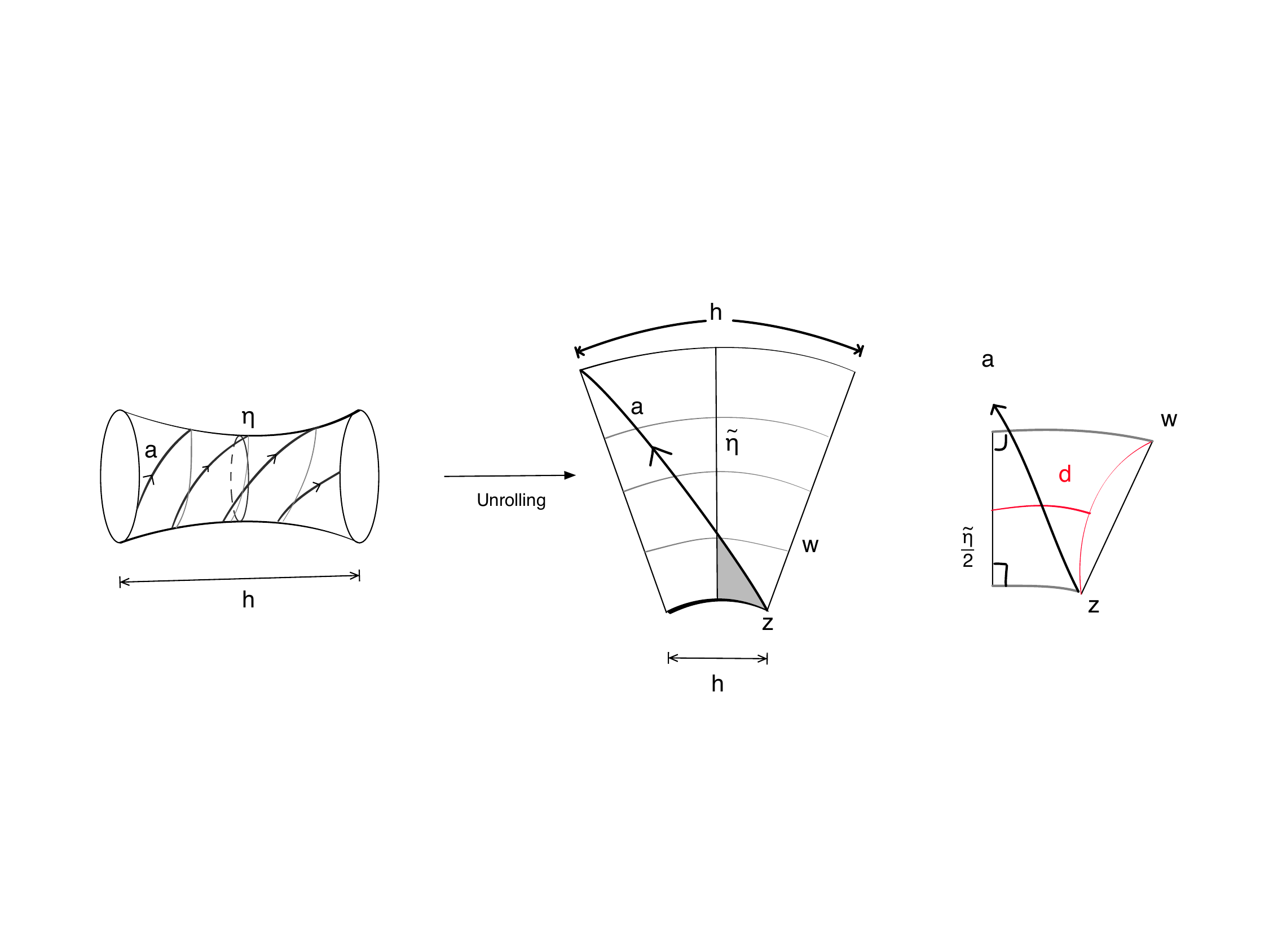}
\vspace{-90pt}

\caption{  Unrolling along $h$ a cylinder containing a geodesic arc $a$ wrapping around $m$-times.}
\label{fig:unrolling}
\end{figure}

We define the function 

\begin{equation}\label{eq: collar winding function}
    f_m (x) = 2\text{arccosh}\left[\coth \left(\frac{x}{2}\right)  \cosh \left(\frac{mx}{2} \right)\right]
\end{equation}

The next Lemma shows that all but two turns of a geodesic arc about a simple closed geodesic  take place in the natural collar about the geodesic. 

\begin{lem}[Collar winding]
\label{lem:winding deep in a collar} Let $\mathcal{C}$ be the maximal collar about $\eta$. 

\begin{enumerate}
\item  Suppose the  geodesic arc $a$ crosses $\eta$ and winds $m$-times around $\eta$ in $\mathcal{C}$. Then all but at most two  of the intersections of $a$  with $\delta$  occur in the natural collar of $\eta$. That is, $a$ winds at least $(m-2)$-times around $\eta$ in the natural collar. 
In particular, the geodesic in the homotopy class of 
$\eta^{m}\ast \gamma_{0}$ winds $(m-2)$ times in the natural collar
$\mathcal{N}$. 

\item If the geodesic arc $a$ crosses $\eta$ and  winds $m$-times in  the natural collar of a geodesic of length $\ell$, then 

\begin{equation}
     f_m (l) \leq a \leq f_{m+1} (l)
\end{equation}

In particular, if $\ell = \ell_{\eta}$
\begin{equation}\label{eq: cyl. wrapping}
\coth  \left(\frac{\ell_{\eta}}{2} \right)\cosh \left(\frac{m \ell_{\eta}}{2} \right)  \leq \cosh \frac{a}{2}  \leq
 \coth \left(\frac{\ell_{\eta}}{2}\right)  \cosh \left(\frac{(m+1)  \ell_{\eta}}{2} \right)
\end{equation}

\end{enumerate}
\end{lem}

\begin{proof}
For the proof of item (1), 
we cut the cylinder $\mathcal{C}$ along $h$ and unroll it to get a
collar neighborhood of the lifted geodesic $\tilde{\eta}$ with $z$ and $w$ consecutive lifts of the intersection point between $h$ and  the right boundary component of $\mathcal{C}$ as in Figure
\ref{fig:unrolling}.
Now join $z$ and $w$ by a geodesic segment and let  $\delta$ be the unique orthogeodesic from the $y$-axis to this geodesic segment as on the right in Figure \ref{fig:unrolling}. The length of $\delta$, call it $d$, satisfies
$$\sinh d \sinh \frac{\ell_{\eta}}{2}<1
\text{  or equivalently } 
d < 
\text{arcsinh} \left(\frac{1}{\sinh \frac{\ell_{\eta}}{2}}\right)
$$
Hence, $\delta$ is contained in the natural collar of $\eta$ which implies the intersection of 
$\delta$ with $a$ is contained in the natural collar. Since the distance from  $\eta$ to a point on $a$ decreases as we move up  $a$,
we may conclude that with the exception of the part of $a$ that goes from $z$ to the intersection point of $\delta$ with $a$, the part of $a$ in the right half of 
$\mathcal{C}$ is contained in the natural right half collar.  Using a similar argument  for the left half of  $\mathcal{C}$, we may conclude that $a$ winds at least $(m-2)$-times around $\eta$ in the natural collar.

For item (2),
 first note that the length of $h$ is 
$2 r\left(\frac{\ell_{\eta}}{2}\right)$, and hence 
 $\cosh \frac{h}{2} =\cosh \left[r \left(\frac{\ell_{\eta}}{2} \right)\right]
=\coth \frac{\ell_{\eta}}{2}$. 
To prove  inequality (\ref{eq: cyl. wrapping}), we cut the cylinder along $h$ and unroll it 
to get a collar neighborhood of the lifted geodesic $\eta$ (see Figure \ref{fig:unrolling}). Denoting  the length of the unrolled cylinder by $x$,  note that by assumption 
\begin{equation}\label{eq: length of x}
m \ell_{\eta}  \leq x \leq (m+1)\ell_{\eta}.
\end{equation}
The unrolled cylinder contains  two isometric right triangles each with side lengths $\frac{h}{2}$, $\frac{x}{2}$, and hypotenuse $\frac{a}{2}$; hence by the hyperbolic Pythagorean Theorem
\begin{equation} \label{eq: p-theorem}
\cosh \frac{a}{2} =\cosh \frac{h}{2} \cosh \frac{x}{2}.
\end{equation}
Inequality (\ref{eq: cyl. wrapping}) now  follows by   using the inequalities in 
(\ref{eq: length of x})   for  $x$  in 
formula  (\ref{eq: p-theorem}),  and the fact that 
  $\cosh \frac{h}{2} =\coth \frac{\ell_{\eta}}{2}$.
\end{proof}

See  \cite{basmajian2025primeorthogeodesics} for 
related results on geodesic winding. 

\begin{lem} \label{lem: calculus}
 Fix $\text{b} >0$ and 
consider the function $f(x)=2 r\left(\frac{x}{2}\right)  +bx$ for $x >0$, 
where $ r\left(x\right)=\textup{arcsinh}\left(\frac{1}{\sinh x}\right)$.
Then $f(x) \geq 2 \left[\textup{arcsinh}(\text{b})+\text{b } \textup{arcsinh}(\frac{1}{b})                 \right]$. The  minimum value occurs at $x=2\textup{arcsinh}(\frac{1}{b}).$
\end{lem} 
Figure \ref{fig:lemma 5.1} depicts the graph $f$.
The proof of lemma \ref{lem: calculus} is a calculus exercise which we leave to the reader.  

\begin{figure}
\includegraphics[scale=0.4]{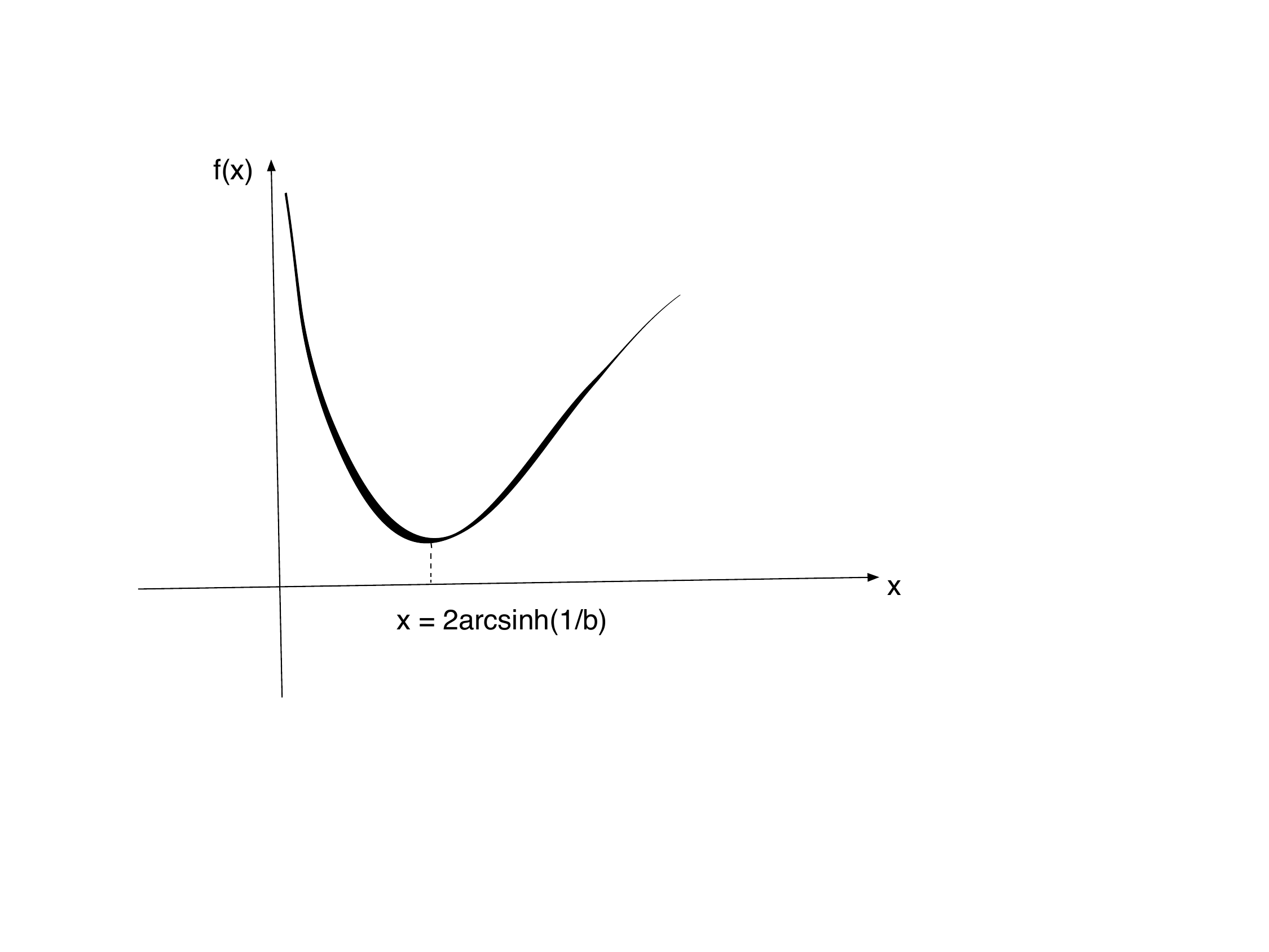}
\vspace{-90pt}
\caption{Graph of $f(x)$ for a fixed b }
\label{fig:lemma 5.1}
\end{figure}

Finally we finish this section with  bounds that will be needed  
later.

\begin{lem} \label{lem: difference bounded} 

Let 
$f_{m}(x)=
 2\textup{arccosh}\left[\coth \left(\frac{x}{2}\right)  \cosh \left(\frac{mx}{2} \right)\right]$.  Then

\begin{enumerate}
\item   
$
2 r(\frac{x}{2})+xm-4\log 2 \leq f_{m}(x)
\leq    2r(\frac{x}{2})+xm+4\log 2$.

\item 
For any $s>0$  and $\epsilon >0$ there exists a positive constant $C$ depending only on $\epsilon$ and $s$ so that  
  $$ f_{m+s}(x)- f_{m}(x) \leq  C$$
  for all $x <\epsilon$ and all $m$ a positive integer.
  
  \end{enumerate}
\end{lem}

\begin{proof} Item (1) follows from elementary bounds
on $\text{arccosh}$ which we leave to the reader. 
Similarly for item (2), using elementary bounds on $\text{arccosh}$,  a straightforward calculation yields

  \begin{align*}
 f_{m+s}(x)- f_{m}(x) &\leq
 2\log 2 +
 2\log {\frac{\coth \left(\frac{x}{2}\right)  \cosh \left(\frac{(m+s)x}{2} \right)}{\coth \left(\frac{x}{2}\right)  \cosh \left(\frac{mx}{2} \right)}}  \\
  &=2\log 2 +
 2\log {\frac{\cosh \left(\frac{(m+s)x}{2} \right)}{\cosh \left(\frac{mx}{2} \right)}} \\
 &=2\log 2 +
 2\log {\frac{\cosh (\frac{mx)}{2} 
 \cosh (\frac{sx}{2})+\sinh (\frac{mx)}{2} )
 \sinh (\frac{sx}{2})}{\cosh \left(\frac{mx}{2} \right)}} \\
 &\leq 2 \log 2 + 2\log \left[
 \cosh (\frac{sx}{2})+
 \sinh (\frac{sx}{2})\right].\\
\end{align*}
Thus for $x<\epsilon$ and $s$ fixed the above expression is upper bounded. 
\end{proof}

\begin{rem}
To summarize, Lemmas \ref{lem:winding deep in a collar} and  \ref{lem: difference bounded} tell us that a geodesic arc that intersects a simple closed geodesic 
$\eta$ and winds $m$-times around it 
\begin{itemize}
\item does almost all of its winding in the natural collar about $\eta$
\item  up to an additive constant, has  length  in the natural collar of $\eta$  which is the sum of $m$-times the length of $\eta$ and the width of the natural collar about $\eta$.
\end{itemize}
\end{rem}


 \section{Constructing curve  families and their coarse length bounds} \label{sec: curve family}

For each closed surface of genus $\geq 2$, we   construct an infinite  family of filling curves   which depend on a base filling curve and a pair of integers $(m,n)$ (See Figure \ref{fig:alpha0fillingclosedsurface}).  The base curve corresponds to $(1,0)$. After constructing the $(m,n)$-curve we compute its intersection number and give coarse bounds on its length. Each such closed surface has at least one  separating curve, allowing us to describe a surgery process along this separating curve. 

\begin{figure}[htp!]

\includegraphics[scale=.35]{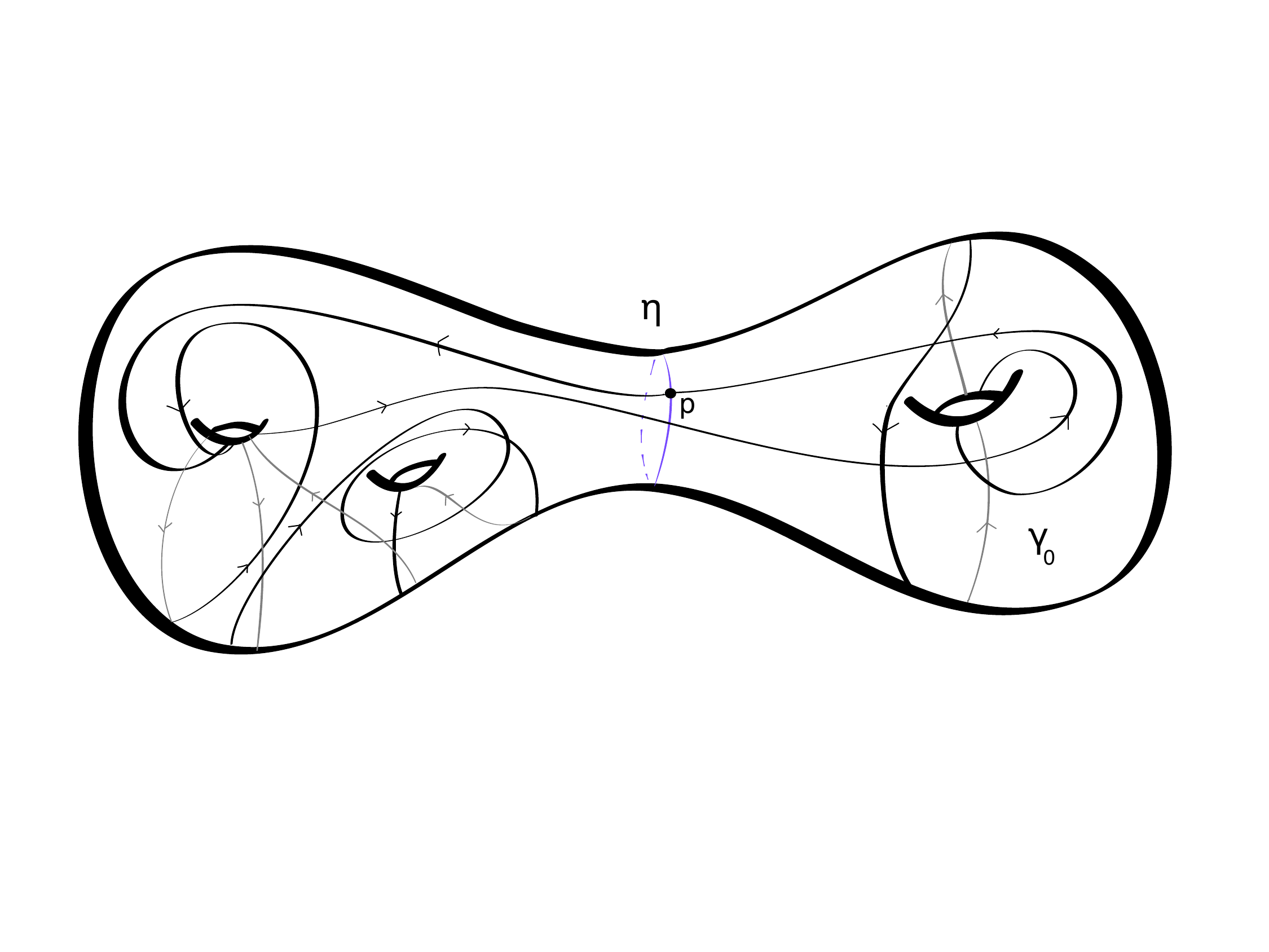}
\vspace{-80pt}
\caption{$\gamma_{0}$ filling on a closed  surface}
\label{fig:alpha0fillingclosedsurface}
\end{figure}

\subsection*{Surgery along a separating simple curve:}

Let $\gamma_0$  be the minimal filling curve from Lemma \ref{lem:minimalandseparating}\and $\eta$  be the  separating curve intersecting it twice. Hence, $i(\gamma_0, \gamma_0)=  2g-1 $, and $i(\gamma_0, \eta)=  2$. We next describe a surgery construction along $\eta$. The curve $\eta$ separates $\Sigma$ into two components, call them $\Sigma_{1}$ and $\Sigma_{2}$.  Since   $i(\gamma_0,\eta) = 2$, 
$\gamma_{0} \cap \Sigma_{i}$ (for $i=1,2$) is a connected arc in $\Sigma_{i}$ from $\partial \Sigma_{i}$ to itself.  Let $p$   be  an intersection point of  
$\gamma_0  \cap \eta$, and consider $\gamma_0$ and 
$\eta$ as elements of $\pi_1 (\Sigma,p)$. Given a pair of positive integers $(m,n)$, let $\gamma_{m,n} \in \pi_1 (\Sigma,p)$
be the curve $\eta^{m}\ast  \gamma_{0}^{n}$, read from right to left (see Figure \ref{fig:alphaclosedsurface} for an example).
Let $X$ be a hyperbolic structure on $\Sigma$. 
 Recall that, $X_i$ is the hyperbolic structure $X$ restricted to $\Sigma_i$.

\begin{figure}

\includegraphics[scale=0.35]{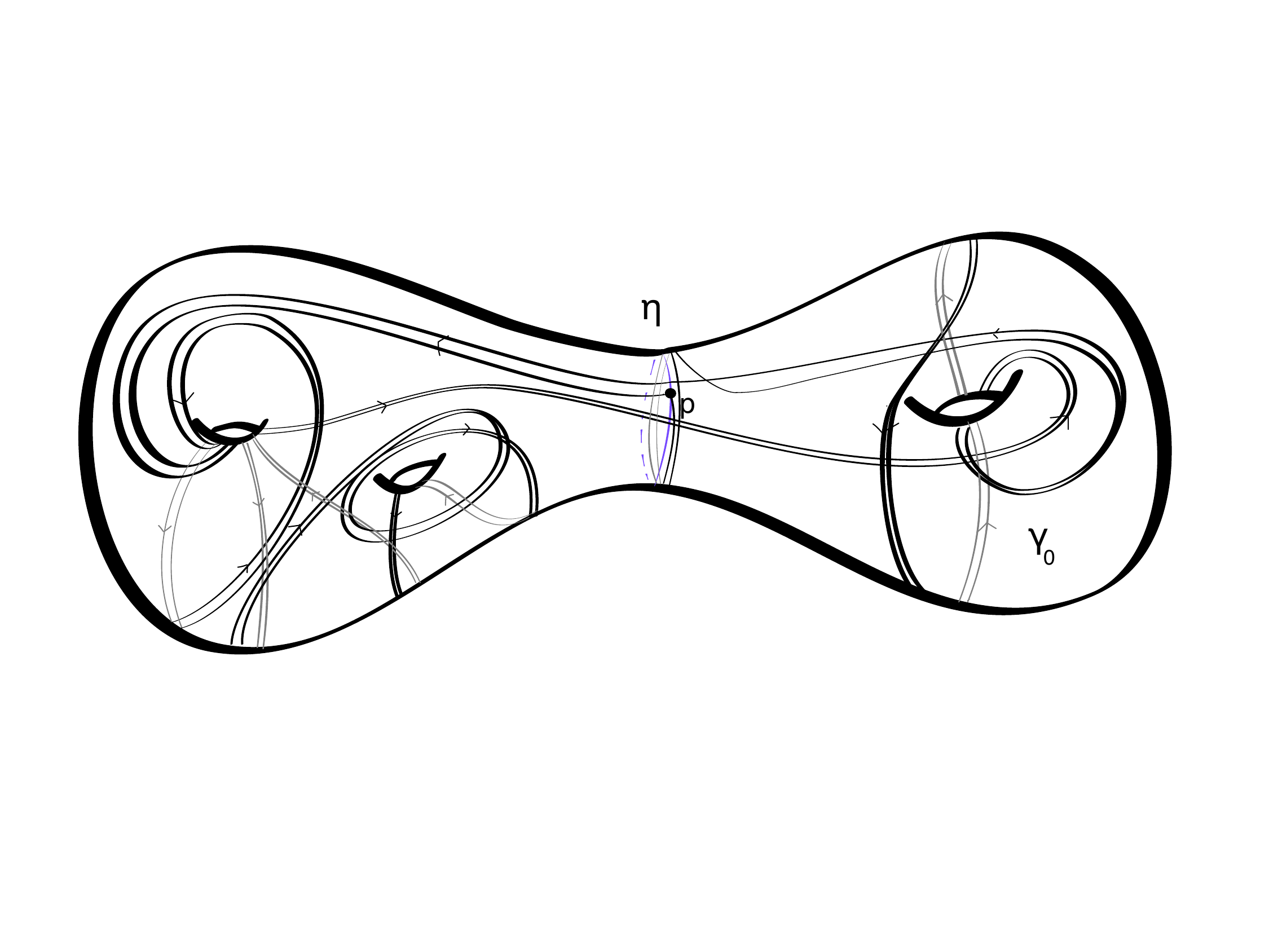}
\vspace{-90pt}
\caption{$\eta^{2}\ast  \gamma_{0}^{2}$ on a closed  surface}
\label{fig:alphaclosedsurface}
\end{figure}

\begin{prop} \label{prop: separating}
Given a pair of positive integers $(m,n)$, let $\gamma_{m,n}= \eta^{m}\ast  \gamma_{0}^{n}$, where $\gamma_0$ and $\eta$ are the curves in $\Sigma$ as above.  
Then 
\begin{enumerate}
\item  $i(\gamma,\gamma)=
n^{2}(2g - 1) + 2mn - m$.
\item   For any  $X$  contained 
in  $\mathcal{T}(\Sigma)$ and $m \geq 2$
$$ 2 \, \text{n} \bigg[ \, \text{r} \left(\frac{\textup{sys}(X_{1})} {2}\right)+ \, \text{r} \left(\frac{\textup{sys}(X_{2})} {2}\right)\bigg]   +  ( 4\text{n}-2) \, \text{r} \left(\frac{\ell_{\eta} (X)} {2}\right) + f_{m-2}(\ell_{\eta}(X)) \leq \ell_{\gamma_{m,n}}(X) $$

 \item   For any $X \in T_{\epsilon}(\Sigma)$

\begin{align*}
 \begin{split}
    \ell_{\gamma_{m,n}}(X) \leq & 
 (n-1)\left[2D+ 2K +4\, \text{r}\left(\frac{\ell_{\eta}(X)}{2}\right)   \right]\\
    &+\left[2D+ 3K +2\, \text{r}\left(\frac{\ell_{\eta}(X)}{2}\right)   \right]+ f_{m+1}(\ell_{\eta}(X))
    \end{split}
 \end{align*}

\end{enumerate}
where $f_{m}(x)$ is the function defined in formula 
(\ref{eq: collar winding function}),  the constant $D$ depends only on  $\epsilon$, and $K$ is a universal constant.
\end{prop}

\begin{figure}[htp!]
\includegraphics[trim=1cm 6cm 1cm 1cm, clip=true, width=0.7\textwidth]{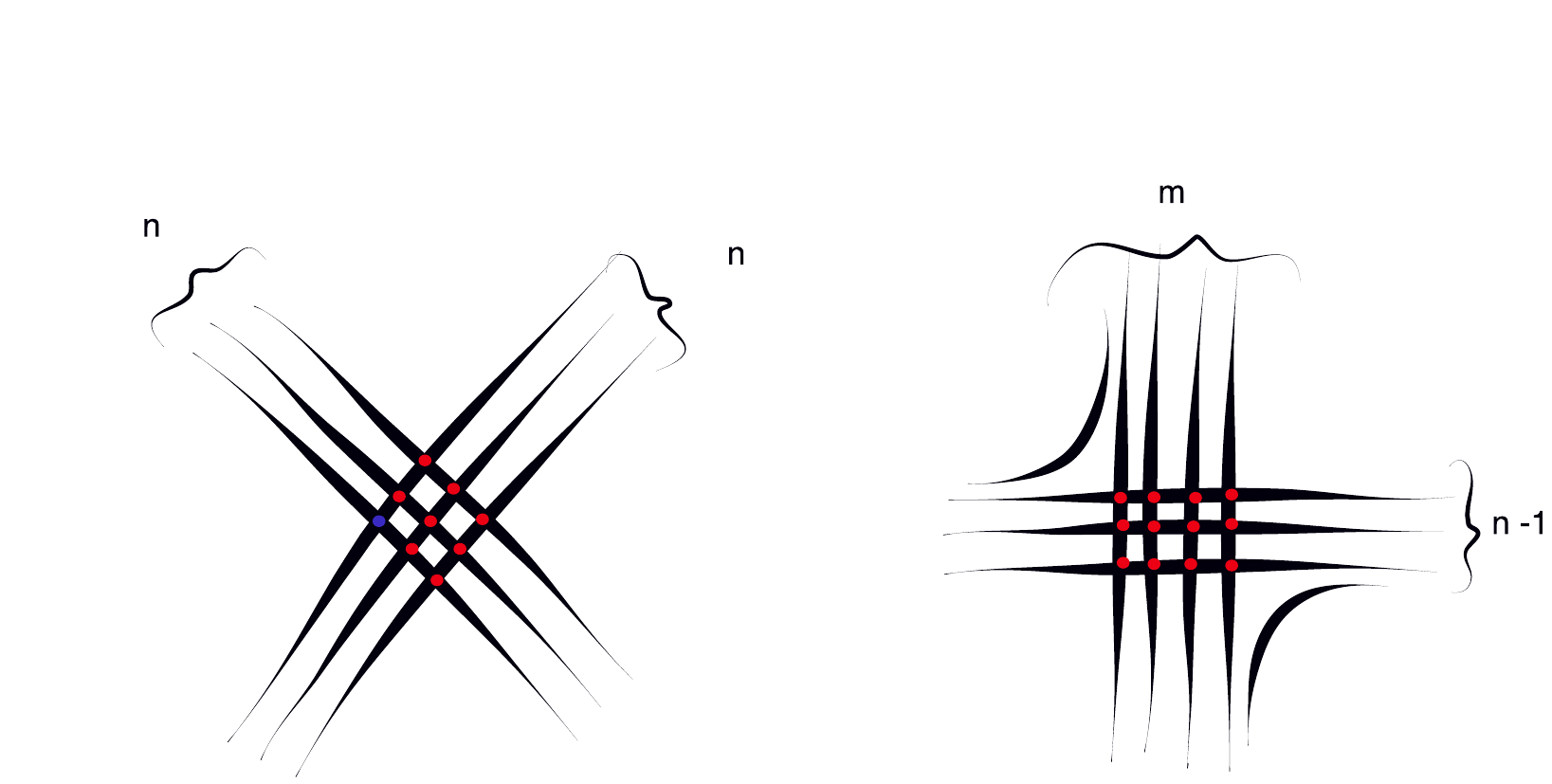}
\vspace{10pt}
\caption{Local self-intersection patterns of $\gamma_{m,n}$.}
\label{fig:intersection pattern}
\end{figure}
\begin{proof}  Item (1) follows from the local intersection patterns that arise 
from the original intersections of $\gamma_{0}$ and the new intersections between $\gamma_{0}$ and $\eta$.
A self-intersection point for $\gamma_{m,n}$ arising  from the original curve  $\gamma_{0}$ is depicted on the left of  Figure \ref{fig:intersection pattern}. New intersection points that arise in a neighborhood of $\eta$ are illustrated on the right of Figure \ref{fig:intersection pattern}.

Proof of Item (2): For ease of notation, we denote the $X$-geodesic in the homotopy class of $\gamma_{m,n}$ by $\gamma$. Let $X$ be in $\mathcal{T}(\Sigma)$. Now, we cut $X$ up in 3-pieces, one being the natural collar $\mathcal{C}$ about the $X$-geodesic in the homotopy class of $\eta$, and the others being the two thick components in $X - \mathcal{C}$.  Set  $a_i' = X$-length of $\gamma$ in $X$ restricted to the thick pieces for $i= 1,2 $ and $a_3 = X$-length of $\gamma$ in the natural collar about $\eta$ (See Figure \ref{fig:separating case split}).
Now, $\gamma$ winds $m$-times around $\eta$, and hence by Lemma \ref{lem:winding deep in a collar}, $\gamma$ winds at least $m-2$ times in $\mathcal{C}$. Also, there are $2n-1$-strands in $\mathcal{C}$, each of length at least $2 \text{r} \left( \frac{\ell_{\eta}(X)}{2}\right )$. So, $a_3\geq  (2\text{n} -1) 2\text{r} \left( \frac{\ell_{\eta}(X)}{2}\right ) + f_{m-2}(\ell_{\eta}(X))$.
$\gamma$ has $n$-strands in each of the thick parts and each of these strands must cross $\text{sys}(X_i)$, ${i = 1,2}$. 
Moreover, since the natural collar of a systole in $X_i$ is disjoint from the natural collar of $\eta$, we may conclude $a_i' \geq 2 \text{r} \left(\frac{\text{sys}(X_i)}{2} \right)\text{n}$ for ${i = 1,2}$.

So,
\begin{align}
    \begin{split}
a_1' + a_3 + a_2' \geq \bigg[2 \, \text{r} \left(\frac{\textup{sys}(X_{1})} {2}\right) &+ 2 \, \text{r} \left(\frac{\textup{sys}(X_{2})} {2}\right)\bigg] \text{n} \\
&+ (2 \,\text{n}-1) \,2 \, \text{r} \left(\frac{\ell_{\eta} (X)} {2}\right) + f_{m-2}(\ell_{\eta}(X)).\\
        \end{split}
\end{align}
To prove item (3), we continue to denote $\gamma_{m,n}$ by $\gamma$ and parametrize  $\gamma : [0,1] \rightarrow X$ so that 
$\gamma(0)=\gamma(1)=p$. Assume  $X \in \mathcal{T}_{\epsilon}(\Sigma)$.
The curve $\eta$ divides $\Sigma$ into two components we call $\Sigma_{1}$  and $\Sigma_{2}$. 
Now for each $i=1,2$, $\gamma \cap \Sigma_{i}$ is a multi-arc  with endpoints in $\eta$. Denote the homotopy class of  this arc (rel $\eta$) by  $\delta_{i}$   in 
$\Sigma_{i}$.
 We denote by $\ell_{\delta_i}(X)$, the $X$-length of the orthogeodesic in the homotopy class of $\delta_i$.
Let $\delta_i'$ be $\delta_i$ with its initial and final segment in the natural collar about $\eta$ cut-off. Let $\mathcal{C}$ be this collar and denote the two boundary components of $\mathcal{C}$ as $\partial_{+} \mathcal{C}$ and $\partial_{-} \mathcal{C}$ .  Hence, $\ell_{\delta_i} = \ell_{\delta_i'} + 2\text{r} \left (\frac{\ell_{\eta}}{2} \right )$.
 
 Now, the $X$-geodesic length of $\gamma$ is at most as long as the following piecewise smooth curve which is homotopic to it. In the following, we construct this piecewise curve step-by-step, where the pieces of the curve are made of $X$-geodesics and segments of the boundary of $\mathcal{C}$ (See Figure  \ref{fig:constructing piecewise curve}). Next to each step, we provide the length bound on that piece.  Let $q$ be the intersection point of $\gamma_0$ with $\partial_{+} \mathcal{C}$.

 \begin{center}\begin{table}

\begin{tabular}{ | m{1em} | m{10 cm}| m{2.5cm}| } 
		\hline \hline 
		{\bf} & {\bf Pieces} &{\bf Length} \\
		\hline
		\hline 
		\hline
		1&Start at $q \in \partial_{+}\mathcal{C}$ and travel along $\partial_{+} \mathcal{C}$ to get to the starting point of $\delta_1'$  &$\leq K $ \\
		\hline
		2& Travel along $\delta_1'$ & $\leq D$ \\
        \hline
         3& Travel along the orthogeodesic from $\partial_{+}\mathcal{C}$ to $\partial_{-}\mathcal{C}$  starting at endpoint of $\delta_1'$& $2 \text{r} \left (\frac{\ell_{\eta}}{2} \right ) $ \\
         \hline
         4 & Travel along $\partial_{-} \mathcal{C}$ to the start of $\delta_2'$ & $\leq K $ \\
         \hline
         5 & Traverse along $\delta_2'$ & $\leq D$\\
         \hline
         6 & For $n=1$ go to step 8; for $n > 1$ from endpoint of $\delta_2'$ traverse orthogeodesic from $\partial_{-}\mathcal{C}$ to $\partial_{+}\mathcal{C}$ & $2\text{r} \left (\frac{\ell_{\eta}}{2} \right )$\\
         \hline
         7 & Repeat steps 1 through 6 $(n-1)$-times & \\
         \hline
         8 & After the  $n^\text{th}$ time finishing on $\partial_{-}\mathcal{C} $, wrap around $\eta $ $m$-times and end at $\partial_{+}\mathcal{C}$& $f_{m+1}(\ell_{\eta}(X))$ \\
         \hline
        9 & Might possibly have to traverse along $\partial_{+}\mathcal{C}$ to get back to $q$ &$\leq K$\\
         \hline
        
        \end{tabular}
    
         \caption{Piecewise geodesic}
         \label{table:2}
        \end{table}
         \end{center}

\begin{figure}

\includegraphics[scale=0.4]{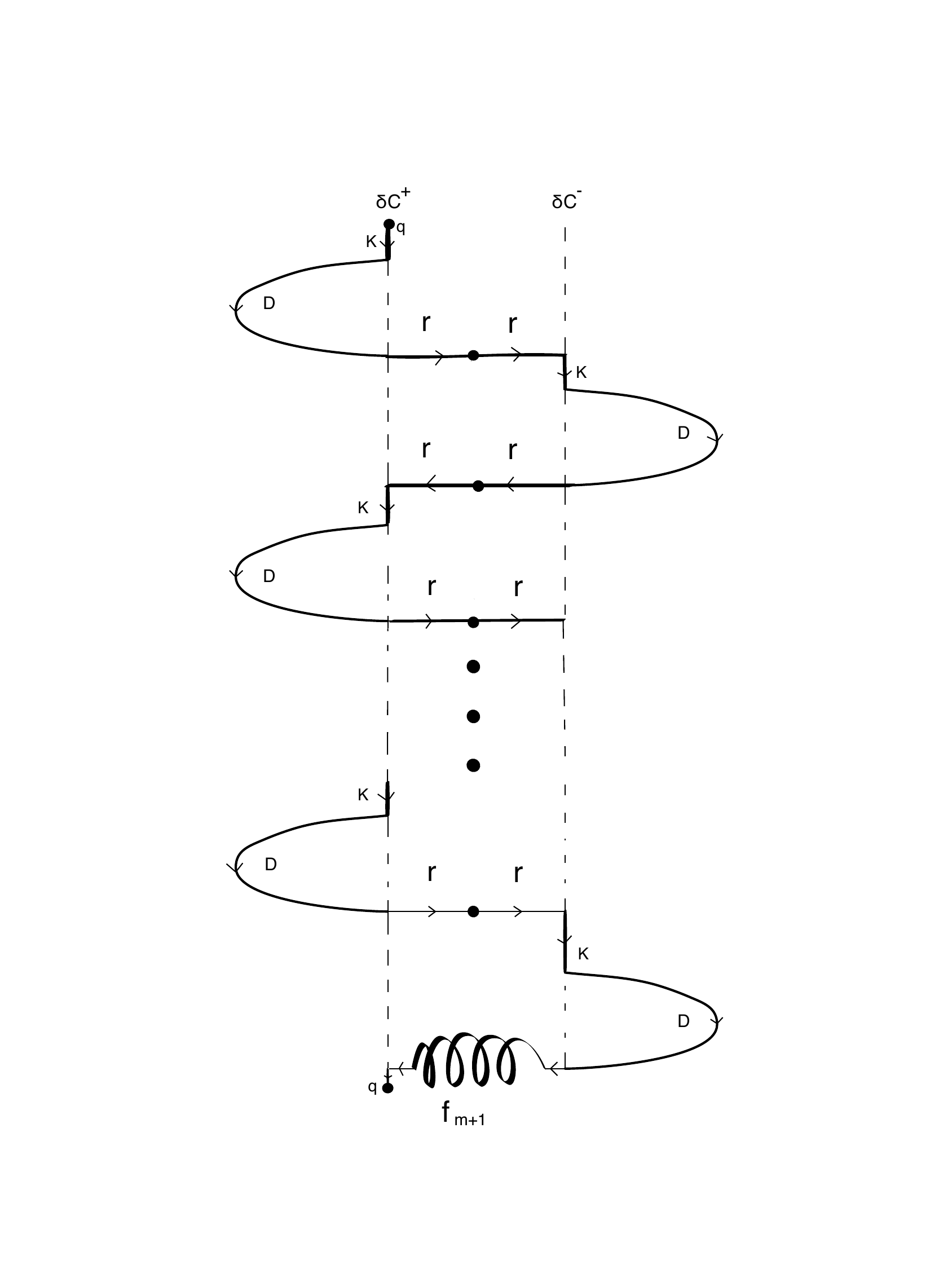}
\vspace{-30pt}
\caption{Constructing a piecewise curve homotopic to $\gamma$.}
\label{fig:constructing piecewise curve}
\end{figure}

 The upper bound in item (3) can be obtained by summing the lengths of the pieces in the steps described in Table \ref{table:2}.
Steps (1) through (7) contribute:  $$(n-1)\left( \left[2D+ 4r\left(\frac{\ell_{\eta}(X)}{2}\right)+2 K\right]\right )$$
where $D$ is the constant that appears in Lemma  $\ref{lem: arc bounded}$, and $K$ is a universal constant arising from the fact that  the length of the boundary of a collar about $\eta$ is upper bounded since 
$X \in \mathcal{T}_{\epsilon}(\Sigma)$.

For the $n^\text{th}$ time we repeat steps (1)-(5), which contribute $$ \left(\left[2D+ 2r\left(\frac{\ell_{\eta}(X)}{2}\right)+2 K\right]\right ).$$ 

We are now at a point on $\partial_{-} \mathcal{C}$. We finish by wrapping around $\eta$ $m$-times and ending at 
$\partial_{+} \mathcal{C}$ (step 8).
Finally,  we move along $\partial_{+} \mathcal{C}$ to the start point $q$ (step 9).

Adding all the contributions, gives us 
 
 \begin{align*}
 \begin{split}
    \ell_{\gamma_{m,n}}(X) \leq & 
 (n-1)\left[2D+ 2K +4\, \text{r}\left(\frac{\ell_{\eta}(X)}{2}\right)   \right]\\
    &+\left[2D+ 2K +2\, \text{r}\left(\frac{\ell_{\eta}(X)}{2}\right)   \right]+ f_{m+1}(\ell_{\eta}(X)) + K.
    \end{split}
 \end{align*}
   
\end{proof}

\begin{rem}
    Our construction can be adapted to construct other families of filling curves by starting with a different filling curve or a different separating curve.
\end{rem}

\begin{lem}\label{lem:etagoestozero} 
Let $\gamma_{m,n} \in \pi_1 (\Sigma,p)$
be the curve $\eta^{m}\ast  \gamma_{0}^{n}$. Then
\begin{enumerate}
\item  For $n$ bounded above, $\ell_{\eta}(X_{\gamma_{m,n}}) \rightarrow 0$, as $m \rightarrow \infty$.
\item For $m$ bounded above, 
$\liminf_{n \rightarrow \infty} \ell_{\eta}(X_{\gamma_{m,n}}) \geq C$, where $C$ is a constant that only depends on $\gamma_{0}$.
\end{enumerate}
\end{lem}

\begin{proof} For ease of notation in this proof we use $X_{m,n}$ to denote the optimal metric for $\gamma_{m,n}$.
In  the natural collar about 
$\eta$ in  $X_{m,n}$ we obtain the following lower bound 
$\ell_{\gamma_{m,n}}(X_{m,n}) \geq $
\begin{equation}\label{eq: lower bound}
 2(2n-1) r\left(\frac{\ell_{\eta}(X_{m,n})}{2}\right)
+2  r\left(\frac{\ell_{\eta}(X_{m,n})}{2}\right)
+(m-2) \ell_{\eta}(X_{m,n}) -C
\end{equation}
The first term is from the $2n-1$ threads of $\gamma_{m,n}$ that pass through the natural collar, and the final three terms come from the contribution of the final thread that winds $m$ times around $\eta$ (see Lemmas
\ref{lem:winding deep in a collar} and Lemma
\ref{lem: difference bounded}). 

We next derive an upper bound for the  $X_{m,n}$-length of $\gamma_{m,n}$ by using a comparison metric, $X  \in 
\mathcal{T}_{\epsilon}$ for some small fixed $\epsilon >0$. Applying item (3) of Proposition \ref{prop: separating}
and Lemmas \ref{lem:winding deep in a collar} and 
\ref{lem: difference bounded} we obtain the following upper bound on the $X$-length of $\gamma_{m,n}$

\begin{align}\label{al: upper bound}
 \begin{split}
    \ell_{\gamma_{m,n}}(X) \leq & 
 (n-1)\left[2D+ 2K +4\, \text{r}\left(\frac{\ell_{\eta}(X)}{2}\right)   \right]\\
    &+\left[2D+ 3K +2\, \text{r}\left(\frac{\ell_{\eta}(X)}{2}\right)   \right]+ f_{m+1}(\ell_{\eta}(X)) \\
   &\leq  4n \text{r}\left(\frac{\ell_{\eta}(X)}{2}\right)  +2nD+(2n+1)K+(m+1)\ell_{\eta}(X)
   \end{split}
 \end{align}
where in the last inequality we have combined terms and used 
$$
f_{m+1}(\ell_{\eta}(X))\leq (m+1)\ell_{\eta}(X) +2\text{r}\left(\frac{\ell_{\eta}(X)}{2}\right).
$$

Now if $n$ is bounded above,  assume  $\ell_{\eta}(X_{m,n})$ is bounded below for the sake of contradiction. 
  Then   the lower bound  in inequality (\ref{eq: lower bound}) grows at least linearly in $n$.
  On the other hand, for the upper bound we choose a comparison metric $X \in \mathcal{T}_{\epsilon}$ so that   the $X$-length, $\ell_{\eta}(X)=\frac{1}{m+1}$.  Plugging into the  upper bound  (\ref{al: upper bound}) 
  and using the fact that $r(\frac{x}{2}) \sim \log \frac{1}{x}$ as $x$ goes to infinity we see that the upper bound is grows at most like $\log m$. This is a contradiction. Thus 
  $\ell_{\eta}(X_{m,n}) \rightarrow 0$, as $m \rightarrow \infty$. This proves item (1).

  To prove item (2), suppose  $m$ is bounded above and assume for the sake of contradiction that 
  $\ell_{\eta}(X_{m,n}) \rightarrow 0$,  as $n \rightarrow \infty$. Set $f(n):=\ell_{\eta}(X_{m,n})$. Then the lower bound inequality (\ref{eq: lower bound}) grows at least like $4n\log \frac{1}{f(n)}$.  For the upper bound inequality (\ref{al: upper bound}) choose a  comparison $X \in \mathcal{T}_{\epsilon}$ where 
  $\ell_{\eta}(X)=1$. Then the upper bound grows at most like   a linear function in $n$. This is a contradiction for large $n$ proving item (2).
\end{proof}

\section{Designer metrics} \label{sec: designer metrics}

Let $\Sigma$ be a surface of genus $g$ with $g \geq 2$. In Section \ref{sec: curve family} we constructed a family of filling curves $\{\gamma_{m,n}\}$, given by two integer parameters.  Using this family,  we construct an infinite  curve pair family $\{(\alpha_{k},\beta_{k})\}_{k \in \mathcal{K}}$, where 
\begin{itemize}
    \item $\alpha_k$ winds around the minimal filling curve once and the separating curve multiple times
    \item $\beta_k$ winds around the minimal curve multiple times and once around the separating one
    \item $\alpha_k$ and $\beta_k$ have the same self-intersection number $k$. 
\end{itemize}

More precisely their definitions are:

\vskip10pt

\noindent{\bf The $\beta_{k}$ curve:}  For  $m=1$ and  any
$n \in \mathbb{N}$,  denote the curve $\gamma_{1,n}$ by $\beta_{k}$,
where  by item (1) of Proposition \ref{prop: separating}

\begin{equation}\label{eq:intersection k}
k:=i(\beta_{k}, \beta_{k})=(2g-1)n^{2} +2n - 1
 \end{equation}
Now we can express $n$ in terms of $k$ as $$n =\frac{-1 + \sqrt{2g + (2g -1)k }}{2g-1}$$

\vskip10pt
\noindent{\bf The $\alpha_{k}$ curve:} For $n=1$ and 
$m = k-2g+1$, denote the curve $\gamma_{m,1}$ by $\alpha_{k}$.
Using item (1) of Proposition 
\ref{prop: separating} we obtain 

$$i(\alpha_{k}, \alpha_{k})=(2g-1) +2m - m = k$$

Set $$\mathcal{K} = \{(2g-1)n^{2} +2n - 1 : n \in \mathbb{N} \}$$ 
This is the set of all self-intersection  values  $k$ obtained in Equation \ref{eq:intersection k} as $n$ runs through the natural numbers.

In this section, we focus on the optimal  metrics associated with these curves. These so-called designer metrics have geometric features that reflect the behavior of the curves (and hence the geodesics) they represent. 


\begin{thm}[Optimal metrics] 
\label{thm: sepoptimalmetric}

Suppose $ \Sigma$ is a closed surface of genus $g \geq 2$.  Let  $\{(\alpha_{k},\beta_{k})\}_{k \in \mathcal{K}}$ be the infinite collection of curve pairs as defined above. Then

\begin{enumerate}

\item the optimal metrics $\{X_{\beta_{k}}\}_{k  \in \mathcal{K}}$ satisfy
$$ \liminf_{k \rightarrow \infty} sys(X_{\beta_{k}}) 
\geq 2 r\left(\frac{\ell_{{\gamma}_{0}}(X_{\gamma_0})}{2}\right).$$
In particular, the metrics $\{X_{\beta_{k}}\}_{k  \in \mathcal{K}}$
are contained in a  compact subspace of $\mathcal{M}(\Sigma)$.

\item  There is a constant $C=C(\gamma_0)>0$ so that   
\begin{displaymath}
\textup{sys}\left(X_{\alpha_{k}}-\{\eta\}\right) \geq C \text{ for  any  $k  \in \mathcal{K}$}.
\end{displaymath}
Moreover,  The optimal metrics $\{X_{\alpha_{k}}\}_{k  \in \mathcal{K}}$ limit  to the stratum $\mathcal{S}$ in $\partial\mathcal{M}(\Sigma)$ which
 correspond  to  $\eta$ being pinched.
\end{enumerate}

\end{thm}

\begin{proof}[Proof of Theorem \ref{thm: sepoptimalmetric}]

    To prove item (1),  first recall that $\beta_k=\eta \ast  \gamma_{0}^{n}$ where 
$n= \frac{-1 + \sqrt{2g + (2g -1)k}}{2g-1}$, 
and $\beta_k$ crosses every simple closed curve $n \approx \sqrt{k}$ times. Now to simplify notation set  
$s_1=sys(X_{\beta_k})$.  We have

\begin{equation}\label{eq: lower bound with epsilon}
m_{\beta_k}:=\ell_{\beta_k} (X_{\beta_k}) \geq 2n r\left(\frac{s_1}{2}\right) = 2 \left(\frac{-1 +\sqrt{2g + (2g -1)k}}{2g-1} \right)r\left(\frac{s_1}{2}\right).
\end{equation}

On the other hand, 
\begin{align}
m_{\beta_k} 
=\ell_{\beta_k} (X_{\beta_k})\leq \ell_{\beta_k} (X_{\gamma_0}) \leq 
 n\ell_{\gamma_0} (X_{\gamma_0})+\ell_{\eta}(X_{\gamma_0})
 \\  \leq 
 \left(\frac{-1 +\sqrt{2g + (2g -1)k}}{2g-1}\right) 
 \ell_{\gamma_0} (X_{\gamma_0})+\ell_{\eta}(X_{\gamma_0})
\end{align}
where the last inequality follows from the fact that 
$\beta_k=\eta \ast  \gamma_{0}^{n}$ and 
$$n= \frac{-1 +\sqrt{2g + (2g -1)k}}{2g-1}.$$ 

Putting these inequalities together we have 

\begin{equation}\label{eq: double inequality}
2 n r\left(\frac{s_1}{2}\right)
\leq m_{\beta_k}  
\leq 
n
 \ell_{\gamma_0} (X_{\gamma_0})+\ell_{\eta}(X_{\gamma_0})
\end{equation}

Solving for $s_1$ in (\ref{eq: double inequality}) gives the lower bound

$$
\textup{sys}(X_{\beta_k}) 
:= s_{1}\geq  2 r \left(\frac{\ell_{{\gamma}_{0}}(X_{\gamma_0})}{2} 
+\frac{\ell_{\eta}(X_{\gamma_0})}{2\left(\frac{-1 +\sqrt{2g + (2g -1)k}}{2g-1}\right) } \right).
$$

Thus, 
$$\liminf_{k \rightarrow \infty} \textup{sys}(X_{\beta_k}) 
\geq 2 r \left(\frac{\ell_{{\gamma}_{0}}(X_{\gamma_0})}{2}\right)
$$
Finally, the Mumford compactness theorem implies that the metrics 
$\{X_{\beta_{k}}\}_{k  \in \mathcal{K}}$ are contained in a compact subspace of moduli space \cite{farb2011primer}.

\begin{figure}
\includegraphics[width=.87\linewidth]{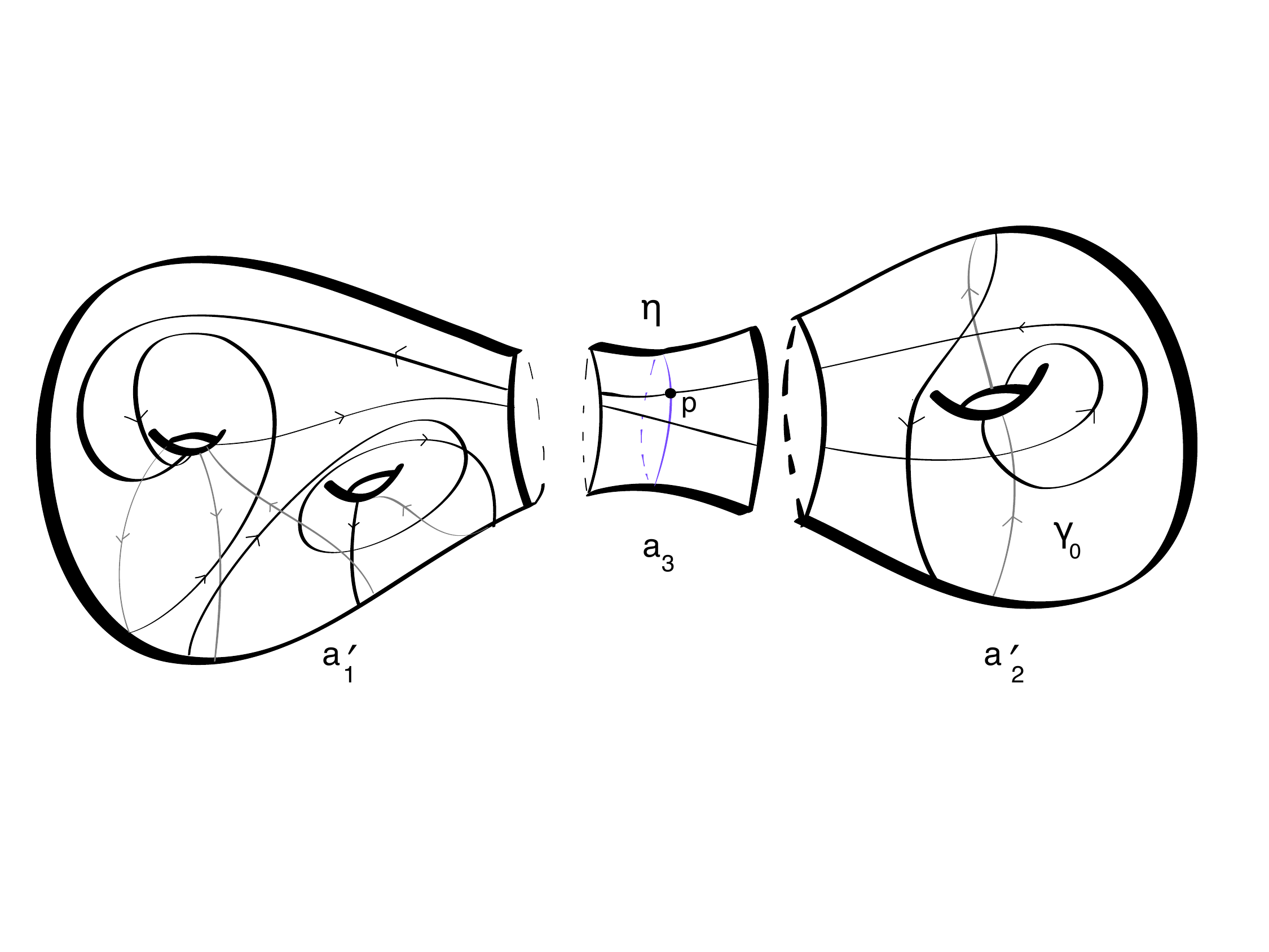}
\vspace{-80pt}
\caption{Thick-thin decomposition }
\label{fig:separating case split}
\end{figure}

To prove item (2),  let $X_{1}$ and $X_{2}$ be the left and right components of $X_{\alpha_{k}}-\{\eta\}$.
We need to show  that outside the natural collar of  
$\eta$  the optimal metric for the $\alpha_{k}$-curve  has  a lower bound on its systole which is independent of $k$ and only depends on the initial filling curve  $\gamma_{0}$. To this end,  we partition the 
$X_{\alpha_{k}}$-geodesic in the homotopy class of $\alpha_{k}$ 
into 3 geodesic segments  $a_1^{\prime}, a_3, a_2^{\prime}$ as in Figure \ref{fig:separating case split}. The geodesics  $a_1^{\prime}$ and $a_2^{\prime}$  are in the so-called thick part of $X_{\alpha_{k}}$, and 
$a_3$  in the so-called thin part.
 By abuse of notation, we will use the same symbols to denote the lengths of these parts. Thus, 
$\ell_{\alpha_k}(X_{\alpha_k})= a_{1}^{\prime}$+ $a_{3}$ + $a_{2}^{\prime}  $.

First we get bounds on $\alpha_{k}$ in the thick part
of $X_{\alpha_{k}}$. We know from Lemma \ref{lem:etagoestozero}, that $\ell_{\eta}(X_{\alpha_k})$ for $k$ large is  $ \leq 1$. Assuming $k$ is large, choose $\epsilon >0$, so that
$X_{\alpha_k} \in \mathcal{T}_{\epsilon}(\Sigma)$.

Next,  choose a comparison metric $X \in T_{\epsilon}(\Sigma)$
and let $\delta^{\prime}_{1}$ be the orthogeodesic from the boundary of the $\eta$-collar to itself in the same relative free homotopy class of $a_1^{\prime}$. Then note  that $a_{1}^{\prime}$ is less than the length of the piecewise curve given by going from the start point of $a_{1}^{\prime}$ along the boundary of the collar to the start point of $\delta_{1}^{\prime}$, followed by doing $\delta_{1}^{\prime}$,  and then going back to the end point of $a_1'$ along the boundary of the collar.
Now, since $\ell_{\eta}(X)\leq {1}$, the length of the boundary of the $\eta$-collar is bounded from above by a positive constant $E$. This is because, $E \rightarrow 2$ as $\ell_{\eta} \rightarrow 0$.
Putting these facts together and using Lemma
\ref{lem: arc bounded}, we  conclude that $a_1' \leq D_1 + E$. Similarily, $a_2' \leq D_2 + E$  where the constants $D_1$ and $D_2$ depend only on $\gamma_0$. Putting these upper bounds together we  have that the length of $\alpha_{k}$ in the thick part satisfies
\begin{equation}\label{eq:thickpart upper}
a_{1}^{\prime} + a_{2}^{\prime} \leq D_0
\text{    where } D_0=2E + D_1 + D_2.
\end{equation}

For ease of notation, set $s_2 :=\text{min} \big\{\textup{sys} (X_{1}),
\textup{sys}(X_{2})\big\}$.
To realize  a lower bound on $a_{1}^{\prime} + a_{2}^{\prime}$, note that
since $a_{1}^{\prime}$ or $a_{2}^{\prime}$ must cross the systole,  by the collar lemma we have
\begin{equation}\label{eq:thickpart lower}
a_{1}^{\prime} + a_{2}^{\prime}
>2r\left(\frac{s_2}{2}\right).
\end{equation}

We next get bounds on $\alpha_{k}$ in the thin part,
that is on $a_3$. Recall,  that $a_3$ has two strands, one which winds 
$m=k-2g+1$ times around $\eta$ and the other that passes through the natural collar of $\eta$.
Then, by the Collar lemma and Lemma \ref{lem:winding deep in a collar}, for any $X \in \mathcal{T}_{\epsilon}$ we have

\begin{equation} \label{eq:thinpart bound}
  2 r \left( \frac{\ell_{\eta}} {2}\right) + f_{m-2}(\ell_{\eta})\leq 
 a_3  
 \leq
 f_{m+1}(\ell_{\eta}) + 2 r \left(\frac{\ell_{\eta}} {2}\right) \text{  where}
\end{equation}
\begin{equation}
 f_{m}(x)=
 2\text{arccosh}\left[\coth \left(\frac{x}{2}\right)  \cosh \left(\frac{mx}{2} \right)\right]
\end{equation}
Putting together the bounds in the thick and thin parts (inequalities \ref{eq:thickpart upper}, \ref{eq:thickpart lower}, and \ref{eq:thinpart bound}) we obtain

\begin{align}
\begin{split}
  2r\left(\frac{s_2}{2}\right) + 2r\left(\frac{\ell_{\eta}(X_{\alpha_k})} {2}\right) +  f_{m-2}(\ell_{\eta}(X_{\alpha_k})) 
  &\leq 
 \ell_{\alpha_k}(X_{\alpha_{k}}) \\
 & \leq 2r\left(\frac{\ell_{\eta}(X_{\alpha_k})} {2}\right) +
D_0+f_{m+1}(\ell_{\eta}(X_{\alpha_k}))\\
\end{split}
\end{align}

Using the left and right above inequalities and solving for $s_2$ we obtain,

\begin{equation}\label{eq: sys bound}
s_2 \geq 2 r\left( \frac{1}{2}\left[D_0+f_{m+1}(\ell_{\eta}(X_{\alpha_k}))-f_{m-2}(\ell_{\eta}(X_{\alpha_k}))\right]\right)
\end{equation}

Now Lemma \ref{lem: difference bounded} guarantees that there is a lower bound on the right-hand side of inequality (\ref{eq: sys bound}) for large $m$ and $\ell_{\eta}$ bounded above. Combining this with the fact that the constant $D_0$ only depends on 
$\gamma_{0}$ yields a lower bound for $\textup{sys}\left(X_{\alpha_{k}}-\{\eta\}\right)$ whcih only depends on $\gamma_0$.

Finally, 
by Lemma \ref{lem:etagoestozero},  $\ell_{\eta}(X_{\alpha_k})$ goes to zero as $k \rightarrow \infty$, thus we may conclude that the optimal metrics 
$\{X_{\alpha_{k}}\}_{k  \in \mathcal{K}}$ limit  to the stratum $\mathcal{S}$ in $\partial\mathcal{M}(\Sigma)$ corresponding to pinching the curve
$\eta$.
\end{proof}

\section{The inf  invariant and self-intersection}
\label{sec: the inf invariant and self-}

Fix $\Sigma$ a surface of the genus $g \geq 2$. Recall the construction of  the curve pairs $\{(\alpha_{k},\beta_{k})\}_{k \in \mathcal{K}}$ from Section \ref{sec: designer metrics}.
 We will use the coarse bounds constructed for the families of curves $\gamma_{m,n} =\eta^{m}\ast\gamma_{0}^{n}$  in Proposition \ref{prop: separating}, to get bounds on their inf invariants.  

\vskip10pt
\noindent{\bf Lower bound on the $\beta_{k}$ curve:}  Recall that the curve $\gamma_{1,n} = \beta_{k}$, where

$$n =\frac{-1 + \sqrt{2g + (2g -1)k}}{2g-1}$$

Let $X = X_{\beta_k}$ be the optimal metric for $\beta_k$.
Using the collar lemma for the collar of $\eta$ and the fact that there are $2n$ strands passing through the natural collar, we have

\begin{align}\label{eq: mbeta}
\begin{split}
& m_{\beta_{k}}:= \ell_{\beta_{k}}(X) \\
&\geq  2\text{n} \left[2\text{ r}\left(\frac{\ell_{\eta}(X)}{2}\right)\right] \\
&\geq  4\left( \frac{-1 + \sqrt{2g + (2g -1)k}}{2g-1} \right)\text{ r}\left(\frac{\ell_{\eta}(X)}{2}\right)\\
&\geq 4 \left(\frac{\sqrt{2g + (2g -1)k}}{2g-1} \right)c_{} - \frac{2c}{2g -1}\\
&\geq 4 \left(\sqrt{\frac{k}{2g-1}} \right)c_{}  - \frac{2c}{2g -1}\\
\end{split}
\end{align}
  where we have used $$n= \frac{-1 + \sqrt{2g + (2g -1)k}}{2g-1} .$$ 

   In the above inequalities the constant $c$ appears from the following argument. Since $\text{sys}(X_{\beta_{k}})$ has a lower bound that does not depend on $k$ by Theorem \ref{thm: sepoptimalmetric}, the optimal metrics $\{X_{\beta_{k}}\}$ are contained in a compact subspace of moduli space. Hence, $\ell_{\eta}$ being a continuous function has an upper bound on this subspace. Thus 
 $\text{r}\left(\frac{\ell_{\eta}(X)}{2}\right)$ has a lower bound, we denote it by ${c}$. 
  
 \vspace{20pt} 

We have shown

\begin{prop}\label{prop: bounds on beta}
 There exists a constant $c >0$ independent of $k$ so that   $$m_{\beta_{k}}= \ell_{\beta_{k}}(X) \geq 4 \left( \sqrt{\frac{k}{2g-1}} \right)c_{}  - \frac{2c}{2g -1}$$
\end{prop}

\vskip20pt
\noindent{\bf Upper bound on the $\alpha_{k}$ curve:} Recall that the curve $\gamma_{m,1} = \alpha_{k}$,
where  
$$m = k-2g+1 $$

To simplify notation,  we set $Y_{\alpha_k}=X_{\alpha_k}-\{\eta\}$. That is,
$Y_{\alpha_k}$ is the optimal metric $X_{\alpha_k}$ restricted to the complement of the geodesic in the homotopy class of $\eta$.

\begin{prop} \label{prop: bounds on alpha}
Let $\Sigma$ be a closed surface with genus $\geq 2$. There exist positive  constants $c_{*}$ and $c_{**}$  so that 
$$ 2 r\left(\frac{\textup{sys}(Y_{\alpha_k})}{2}\right) + 2 \log (k-2g+1) -c_{**}
\leq m_{\alpha_{k}} \leq   c_* + 6 \log (k-2g+1) ,
 \text{ for $k$ large}.$$
 In particular,  
 \begin{align}
 \textup{sys}(Y_{\alpha_k}) \geq  
 2r \left(\frac{c_* +c_{**}+ 4\log (k-2g+1)}{2}\right)
 \end{align}
\end{prop} 

\begin{proof}   Throughout this proof, as a consequence of Lemma \ref{lem:etagoestozero}, we have chosen  $\epsilon >0$
so that  $X_{\alpha_k} \in \mathcal{T}_{\epsilon}(\Sigma)$ . 
We first prove the upper bound. 
  Let $X$ be a metric in $\mathcal{T}_{\epsilon}(\Sigma)$ where $\ell_{\eta}(X)=\frac{\log m}{m}$. 
Using the upper bound in item (3) from Proposition \ref{prop: separating} we have

\begin{align} \label{eq: alphak upper bound}
m_{\alpha_{k}} \leq  \ell_{\alpha_{k}}(X) 
\leq
2D+ 3K + 2 r\left(\frac{\log m}{2m}\right)+
 f_{m+1}\left(\frac{\log m}{m}\right)
\end{align}
\begin{align}\label{eq: alphak upper bound2}
\leq (2D+3K)+2 r\left(\frac{\log m}{2m}\right)+2r\left(\frac{\log m}{m}\right)+2\left(\frac{\log m}{m}\right)(m+1) +2\log 2
\end{align}
\begin{align}\label{eq: alphak upper bound3}
\leq (2D+3K)+4 r\left(\frac{\log m}{2m}\right)+2\left(\frac{\log m}{m}\right)(m+1) +2\log 2
\end{align}

\begin{align} \label{eq: alphak upper bound4}
\leq  c_* + 6 \log m - 4 \log \log m \leq  c_* + 6 \log m
\end{align}
where inequality (\ref{eq: alphak upper bound2}) uses Lemma \ref{lem: difference bounded}, 
  inequality (\ref{eq: alphak upper bound4})  uses  the fact that $r(\frac{x}{2}) \sim \log \frac{1}{x}$  as $x$ goes  to $0$, and hence there exists an $M >0$ so that  \begin{align}
r\left(\frac{\log m}{2m}\right)\leq  \frac{1}{4}+\log \left(\frac{m}{\log m}\right),
\text{ for } m \geq M
\end{align}
and we have
set $c_{*}$ to be  $(2D+3K+ 2 + 2 \log 2 +1)$  in inequality  (\ref{eq: alphak upper bound4}).
Using the upper bound in inequality (\ref{eq: alphak upper bound4}) and plugging in $k-2g+1$ for $m$ we obtain the desired upper bound.

For the lower bound,  we first observe that $\alpha_{k}$ must cross the collar neighborhoods of the systole of $Y_{\alpha_k}$ and the curve $\eta$. Since $\eta$ is separating,  $\alpha_{k}$ must in addition   wind  around  
$\eta$  up to an additive constant $(k-2g+1)$-times. Accounting for both strands of $\alpha_{k}$ we have
 \begin{align}\label{eq: better lower bound on alpha}
 \begin{split}
 m_{\alpha_{k}} &= \ell_{\alpha_{k}}(X_{\alpha_k})\\
 &\geq
2r\left(\frac{\text{sys} (Y_{\alpha_k})}{2}\right) +2r\left(\frac{\ell_{\eta}(X_{\alpha_k})}{2}\right)
  + (k-2g+1) \ell_{\eta}(X_{\alpha_k}) -c_{**}
  \end{split}
  \end{align}
  
  Now using Lemma \ref{lem: calculus} we can bound the sum of the two middle terms from below by  $2 \log (k-2g+1)$ verifying the desired lower bound. 
   \end{proof}

\vskip10pt

In this section, we have proven the following. 
\begin{thm}
\label{thm: separating curve construction} Suppose  $\Sigma$ is a closed surface of genus $g \geq 2$. There exists an infinite set $\mathcal{K}$ of positive integers  and 
a collection of curve pairs $\{(\alpha_{k},\beta_{k})\}_{k \in \mathcal{K}}$ so that

\begin{enumerate}
\item  $\alpha_k$ and $\beta_k$ are each filling curves

\item $i(\alpha_k, \alpha_k)=i(\beta_k, \beta_k)=k$

\item $m_{\alpha_k}\lesssim \log k < \sqrt{k}\lesssim  m_{\beta_k}$
\end{enumerate}
\end{thm}

\begin{proof} Item (1) is clear {by construction} and item  (2) follows  from Proposition 
\ref{prop: separating}.
The proof of item (3)  follows from Propositions \ref{prop: bounds on beta} and \ref{prop: bounds on alpha}.

 \end{proof}

\begin{rem}
 We remark that the above construction holds for a punctured surface of Euler characteristic less than $-1$. Details are left for future investigation. 
 
\end{rem}

\section{The Inf spectrum and final remarks}
\label{sec: Final Remarks}
In this section, we consider the set of all filling curves  on a topological surface.
\vskip10pt

\noindent \textbf{The Inf Spectrum:}
Recall that for $\gamma$ a filling curve,  its orbit under the mapping class group  is denoted  $[\gamma]$, and  its inf invariant by   $m_{[\gamma]}$. Associated to each topological surface $\Sigma$ of genus $g\geq 2$,  we
 define the {\it \text{inf}  spectrum of } $\Sigma$  to be
 $$
 \mathcal{I}Spec (\Sigma) := \{m_{[\gamma]}: \gamma \text{ filling on } \Sigma \}
 $$

We next provide coarse bounds on the {\it inf spectrum}.

\begin{thm}\label{thm: inf spec bounds}
   For any $L>0$, the set

$$\{ [\gamma]:\gamma \text{ filling on }\Sigma, m_{\gamma} \leq L \}  \text{ is finite.}$$

 More concisely,  there exists $L_{0}>0$ so that 

\begin{equation}\label{eq: coarse bounds}
e^{bL}\leq
|\{ [\gamma]:\gamma \text{ filling on }\Sigma,m_{\gamma}\leq L \}|
\leq e^{ce^{L}} 
\end{equation}
for all  $L \geq L_{0}$.
Where $b,c>0$ are constants that only depend on the Euler characteristic $\chi :=\chi(\Sigma)$.
\end{thm}

\begin{proof}
Given $\gamma$ a filling curve, it is a result of  Theorem 1.2 in \cite{souto2018counting} that there exists a hyperbolic structure   $X$   on $\Sigma$ so that 
$$
\ell_{\gamma}(X) \leq c_{1} \sqrt{i(\gamma,\gamma)}
$$
where $c_{1}$ is an explicit constant that only depends on 
$\chi$.
Now  from the lower bound in Theorem 1.1 of  \cite{souto2018counting} there is an explicit constant
$b=b(|\chi|)>0$  so that 
\begin{align*}
e^{bL} &  \leq  |\{[\gamma] : \gamma  
  \text{ filling on }\Sigma, i(\gamma, \gamma)  \leq
  \frac{L^{2}}{c_{1}^{2}}\}| \\
& \leq |\{ [\gamma]:\gamma \text{ filling on }\Sigma, 
c_{1} \sqrt{i(\gamma,\gamma)} \leq L \}| \\
&\leq 
|\{ [\gamma]:\gamma \text{ filling on }\Sigma, \ell_{\gamma}(X_{\gamma}) \leq L \}|
\end{align*}
where the last inequality follows from 
$m_{\gamma}=\ell_{\gamma}(X_{\gamma})\leq \ell_{\gamma}(X)\leq c_{1} \sqrt{i(\gamma,\gamma)}$.
This proves the left-hand  inequality of expression (\ref{eq: coarse bounds}).

For the right-hand inequality of expression (\ref{eq: coarse bounds}),
 we note that for $\gamma$
a filling curve  it follows from \cite{basmajian2013universal} that 
$$\ell_{\gamma}(X_{\gamma}) \geq \frac{1}{2}
\log{\frac{i(\gamma,\gamma)}{2}}.
$$
Hence
\begin{align*}
&|\{ [\gamma]:\gamma \text{ filling on }\Sigma, \ell_{\gamma}(X_{\gamma}) \leq L \}| & \\
    \leq  &  |\{ [\gamma]:\gamma \text{ filling on }\Sigma, \frac{1}{2}
\log{\frac{i(\gamma,\gamma)}{2}} \leq L \}|& \\
= & |\{ [\gamma]:\gamma \text{ filling on }\Sigma, 
i(\gamma,\gamma) \leq 2e^{2L} \}| \\
\leq & e^{ce^{L}} 
\end{align*}
   where $c=c(|\chi |)$, and the last inequality follows from 
the upper bound in Theorem 1.1 of \cite{souto2018counting}.
\end{proof}

\noindent \textbf{Final remarks:}  


\vskip5pt
\begin{itemize}
    \item \textbf{Question 1:} Consider the two parameter family of curves $\gamma_{m,n} =\eta^{m}\ast\gamma_{0}^{n}$. Do these curves have different inf invariants?
    
    \item \textbf{Question 2:} How do the set of curves  $\{\gamma_{m,n}: {m,n \in  \mathbb{N}}\}$, viewed as currents,  sit inside the space of currents. Is there a natural subspace for which these curves form a lattice.
    
    \item \textbf{Question 3:} Given $k \in \mathbb{N}$, define $N(k)$ to be the number of distinct curves with self-intersection number $k$ having different inf invariants. Determine the growth rate of $N(k)$.

    \end{itemize}



\newpage

\bibliographystyle{plain}
\bibliography{REFERENCES}

\end{document}